%% file: spca-variable-selection-article.tex
\documentclass[noinfoline]{imsart}
\usepackage[utf8]{inputenc}
\usepackage{enumerate}
\RequirePackage[OT1]{fontenc}
\RequirePackage{amsthm,amsmath}
\RequirePackage[numbers]{natbib}
\usepackage{amsmath}
\usepackage{amsfonts}
\usepackage{amssymb}
\usepackage{hyperref}
\usepackage{amsthm}
\usepackage{color}
\usepackage{bbm}
\usepackage{graphicx}
\usepackage{caption}
\usepackage{subcaption}

\newtheorem{lemma}{Lemma}[section]
\newtheorem{proposition}{Proposition}[section]
\newtheorem{corollary}{Corollary}[section]
\newtheorem{theorem}{Theorem}[section]

\newtheorem{condition}{Condition}
\newtheorem{remark}{Remark}
\newtheorem{definition}{Definition}[section]
\newtheorem{assumption}{Assumption}

\numberwithin{equation}{section}

\begin{document}
\begin{frontmatter}
\title{Sharp oracle inequalities for stationary points of nonconvex penalized M-estimators}
\runtitle{Sharp oracle inequalities for stationary points}

\begin{aug}
\author{\fnms{Andreas} \snm{Elsener}\ead[label=e1]{elsener@stat.math.ethz.ch}}
\and
\author{\fnms{Sara} \snm{van de Geer}\ead[label=e2]{geer@stat.math.ethz.ch}}

\runauthor{Elsener and van de Geer}

\affiliation{ETH Z\"urich}

\address{
Seminar f\"ur Statistik\\
ETH Z\"urich\\
8092 Z\"urich \\
Switzerland \\
\printead{e1}\\
\printead{e2}}

\end{aug}

\begin{abstract}
Many statistical estimation procedures lead to nonconvex optimization problems. Algorithms to solve these are often guaranteed to output a stationary point of the optimization problem. Oracle inequalities are an important theoretical instrument to asses the statistical performance of an estimator. Oracle results have focused on the theoretical properties of the uncomputable (global) minimum or maximum. In the present work a general framework used for convex optimization problems to derive oracle inequalities for stationary points is extended. A main new ingredient of these oracle inequalities is that they are \textit{sharp}: they show closeness to the best approximation within the model plus a remainder term. We apply this framework to different estimation problems.
\end{abstract}

\begin{keyword}
\kwd{sharp oracle inequality, sparse corrected linear regression,  sparse PCA, sparse robust regression, stationary points}
\end{keyword}

\end{frontmatter}
\input{introduction}
\input{sharp-oracle-inequality}
\input{applications}
\input{discussion2}
\newpage

\appendix
\input{lemmasectionoracle}
\input{proof-maintheorem}	
\input{subgaussian}
\input{empirical-process-bounds}
\vspace{2cm}

\bibliographystyle{plainnat}
\bibliography{myreferences}

\end{document}

%% file: introduction.tex
\section{Introduction}
\subsection{Background and Motivation}
Nonconvex loss functions arise in many different branches of statistics, machine learning and deep learning. These loss functions entail several advantages from a statistical point of view. For instance, in robust regression, where one requires that the influence function of the loss is bounded, nonconvex losses are widely used. Furthermore, they are unavoidable in areas such as deep learning where they arise as a byproduct of the representation of the data. Despite the exponential increase in methodologies involving nonconvex loss functions, there are still many theoretical questions that need to be answered.

As a matter of fact, the nonconvex optimization problems can usually be solved only via algorithms that guarantee convergence to a so-called stationary point. A stationary point is often not the global minimum. It is almost hopeless to recover the latter. Statistical theory has mostly focused on deriving properties of an incomputable global optimum. We show that under certain circumstances stationary points satisfy sharp oracle results similar to those that were derived for the global optimum.

High-dimensional data (i.e. when the number of parameters to be estimated exceeds the number of observations) represent an additional challenge. A well-established way of tackling this problem is to assume that the number of ``active'' parameters is smaller than the dimension of the parameter space. This assumption is typically called ``sparsity''. Estimators designed under the sparsity assumption are often M-estimators with either an additional constraint or a penalty term. Under convex loss functions these approaches are numerically equivalent. Here we focus on the latter approach. We consider estimators that are composed of a nonconvex differentiable loss and a penalty term. Primarily, the penalty term is chosen to be a ``sparsity-inducing'' norm.

We now describe the structure of the estimators that we are interested in. Let $ Z_1, \dots, Z_n$ be independent observations with values in some space $\mathcal{Z}$ stemming from a distribution depending on $\beta \in \mathcal{C} \subseteq \mathbb{R}^p$. Let $\rho$ be a differentiable possibly nonconvex function such that
\begin{equation}
\rho : \mathcal{C} \times \mathcal{Z} \rightarrow \mathbb{R}.
\end{equation}
The function $\rho$ measures the ``misfit'' that arises by taking the decision $\beta$ in comparison to the given data. 

We define $R_n$ as
\begin{equation}
R_n(\beta) = \frac{1}{n} \sum_{i =1}^n \rho(\beta, Z_i).
\end{equation}

 $R_n(\cdot)$ is named the ``empirical risk''. It is a random quantity as it depends on the random observations $\left\{ Z_i \right\}$. The unknown quantity we are interested in estimating is given by the minimizer of the population version:
\begin{equation}
\beta^0 = \underset{\beta \in \mathcal{C}}{\arg \min} \ R(\beta),
\end{equation}
where $R(\beta) = \mathbb{E} R_n(\beta)$ is the risk.

Consider a norm $\Omega(\cdot)$ on $\mathbb{R}^p$ with dual norm of $\Omega_*(\cdot)$. The subdifferential of the norm $\Omega(\cdot)$ is defined as

\begin{equation}
\partial \Omega(\beta) = \left\lbrace \begin{array}{ll}
\left\lbrace z \in \mathbb{R}^p : \Omega_*(z) \leq 1 \right\rbrace, & \text{if} \ \beta = 0, \\
\left\lbrace z \in \mathbb{R}^p: \Omega_*(z) = 1, z^T\beta = \Omega(\beta) \right\rbrace, & \text{if} \ \beta \neq 0.
\end{array} \right.
\end{equation}
\cite{bach2012optimization}. We consider empirical risk minimization problems of the form
\begin{equation}
\label{eqn:optimargmin}
\hat{\beta} = \underset{\beta \in \mathcal{C}}{\arg \min} \ R_n(\beta) + \lambda \Omega(\beta),
\end{equation}
where $\lambda > 0$ is a tuning parameter that needs to be chosen.

To solve optimization problems of the type given in (\ref{eqn:optimargmin}) one often uses gradient descent algorithms and its modifications. However, algorithms for nonconvex optimization problems typically output a local optimum of the objective function (\ref{eqn:optimargmin}) but not $\hat{\beta}$. In this paper we show that points $\tilde{\beta}$ satisfying
\begin{equation}
\label{eqn:stationary}
(\dot{R}_n(\tilde{\beta}) + \lambda \tilde{z} )^T(\beta - \tilde{\beta}) \geq 0, \ \text{for all} \ \beta \in \mathcal{C},
\end{equation}
where $\tilde{z} \in \partial \Omega(\tilde{\beta})$, enjoy some properties of the (incomputable) estimator $\hat{\beta}$. These points $\tilde{\beta}$ are called stationary points.

We extend a general framework introduced in \cite{geer2015} for convex optimization problems. The key property that is needed is called \emph{two point inequality} in \cite{geer2015}:
\begin{equation}
\label{eqn:twopointineq}
-\dot{R}_n(\tilde{\beta})^T(\beta - \tilde{\beta}) \leq \lambda \Omega(\beta) - \lambda \Omega(\tilde{\beta}).
\end{equation}
Using that $\tilde{z}^T \tilde{\beta} = \Omega(\tilde{\beta})$ and that $\Omega_* ( \tilde{z})  \leq 1$ one can see that the two point inequality is indeed satisfied by points that satisfy inequality (\ref{eqn:stationary}).


Let now $\beta^\star \in \mathcal{B}$ be a non-random vector with $\Vert \beta^\star \Vert_0 = s^\star$. We think of the vector $\beta^\star$ as of a quantity that already ``contains'' some additional structural assumption about the estimation problem such as the number of non-zero entries of the target $\beta^0$. The vector $\beta^\star$ optimally trades off the approximation and estimation errors. In this paper we show that stationary points (i.e. points obeying inequality (\ref{eqn:stationary})) also mimic the behavior of the oracle as the optimum $\hat{\beta}$ does. The oracle inequalities that we derive are typically of the following type:
\begin{align}
\label{eqn:oracleineq}
&\Omega(\tilde{\beta} - \beta^\star)+R(\tilde{\beta}) - R(\beta^0) \nonumber \\
&\leq \underbrace{R(\beta^\star) - R(\beta^0)}_{ \text{approximation error}} + \underbrace{C^2 \lambda^2 s^\star}_{ \text{estimation error}} + 2 \lambda \Omega^-(\beta^\star),
\end{align}
where $C>0$ is a constant not depending on the sample size nor on the dimension of the estimation problem. Inequalities of this kind are also named \textit{sharp} since the constant in front of $R(\beta^\star)$ is $1$. This is particularly important if the approximation error $R(\beta^\star) - R(\beta^0)$ is not small. In addition, we also derive rates of convergence for the estimation error measured in different norms. In addition to the Euclidean norm the estimation error can be measured in the $\Omega(\cdot )$-norm.

\subsection{Related literature}
Nonconvex optimization problems are ubiquitous. The most recent example that makes theoretical understanding of stationary points of nonconvex optimization problems necessary is deep learning. As mentioned at the end of Chapter 4.3 of \cite{Goodfellow-et-al-2016} the majority of the problems in deep learning cannot be solved via convex optimization.

Another prominent area where statistical nonconvex optimization problems arise is represented by mixture models. Typically, the estimators are computed by a version of the Expectation-Maximization (EM) algorithm or by a (coordinate) gradient descent algorithm. Examples for this can be found in  \cite{stadler2010ell1} where a finite mixture of regressions is considered in the high-dimensional setting. An EM-type algorithm is proposed and theoretical guarantees for the \textit{global minimizer} are derived. The question about the statistical properties of stationary points (i.e. what the algorithm actually outputs) is left to future research. In \citet{schelldorfer2011estimation} linear mixed-effects models in the high-dimensional setting are studied. A coordinate gradient descent algorithm is proposed and convergence to a stationary point is proven. Also in this latter work there is a gap between what the numerical algorithm outputs and the statistical properties that are shown to hold for the global minimum. However, the situation in the two mentioned papers is still more involved as the population version of the problem has several stationary points. For EM-type algorithms the work of \cite{balakrishnan2017statistical} is the first that guarantees theoretical properties for estimates of symmetric mixtures of two Gaussians and two regressions.

Several high-dimensional estimation problems related to regression lead ineluctably to nonconvex optimization problems. In \cite{loh2011high} corrected linear regression is studied. Three additional sources of noise that lead to nonconvex estimators are examined. The case of additive noise in the predictors, the case of missing data, and the case of multiplicative noise in the predictors are studied. The population versions of these estimation problems are convex. However, due to the estimators of the population covariance matrices they become nonconvex in the sample version. A gradient descent algorithm is proposed and theoretical properties of the minimum are described.

In a follow-up work \cite{loh2015regularized} give theoretical guarantees for the stationary points of nonconvex penalized M-estimators. Their framework also includes nonconvex penalization terms. However, in contrast to the present work they do not provide sharp oracle inequalities. In \cite{loh2014support} the authors give theoretical guarantees for the support recovery using nonconvex penalized M-estimators. The loss function as well as the penalization term are both allowed to be nonconvex.

As far as robust regression is concerned, the use of nonconvex loss functions is particularly appealing. The main robustness-inducing property that is exploited is the boundedness of the gradient/the Lipschitz continuity of the loss. Estimators involving e.g. the Tukey loss function seem therefore particularly well-suited for this task. \cite{loh2015statistical} gives a general framework for this particular type of regularized M-estimators. The penalty term is allowed to be nonconvex as well.

In \cite{mei2016landscape} a general framework to analyze the theoretical properties of \linebreak $\ell_1$-penalized and unpenalized M-estimators is proposed. The former is necessary for the high-dimensional setting whereas the latter are used for the case where the number of observations exceeds the number of parameters to be estimated. Rates of convergence are derived for stationary points of several statistical estimation problems such as robust regression, binary linear classification, and Gaussian mixtures. In contrast, we only consider the high-dimensional setting and derive sharp oracle inequalities from which the rates obtained in \cite{mei2016landscape} can be recovered. Our framework applies also to different types of penalizing norms other than the $\ell_1$-norm.

The nonconvex optimization problems that are considered in the present work can be subdivided into the following types:
\begin{enumerate}
\item The quantity to be estimated $\beta^0$ is the unique global minimizer of the \emph{convex} risk $R(\beta)$. The source of nonconvexity stems exclusively from the sample optimization problem. This case has been considered for example in \cite{loh2011high}. An example for this type of estimation problems is  the corrected linear regression with additive noise in the covariates. It is discussed in Subsection \ref{ss:correctedlinreg}.
\item The quantity to be estimated $\beta^0$ is (a possibly non-unique) global minimizer of the \emph{nonconvex} risk $R(\beta)$. The risk $R(\beta)$ is convex in an $\ell_2$ neighborhood of the target, i.e. on a set of the form
$$\mathcal{B} = \left\{ \beta \in \mathbb{R}^p : \Vert \beta - \beta^0 \Vert_2 \leq \eta \right\}$$
for some suitable constant $\eta > 0$. This case has been studied in \cite{loh2015regularized} and \cite{loh2015statistical}. An example is binary linear classification in Subsection \ref{ss:binaryclass}.
\end{enumerate}

A parallel line of research is concerned with the inspection of the theoretical properties of nonconvex penalization terms. In \cite{zhang2012general} a general framework for concave penalization terms is established. In general, it is argued that concave penalties reduce the bias that results from convex procedures such as e.g. the Lasso \cite{tibshirani1996regression}. We restrict ourselves to the case of norm penalized estimators.

\subsection{Organization of the paper}
In Section \ref{s:sharporacle} we review the notion of an oracle and discuss the additional properties related to the penalization term that are needed for the sharp oracle inequality. The sharp oracle inequality given in Theorem \ref{thm:oracle} is purely deterministic. In Section \ref{s:applications} we show how the (deterministic) sharp oracle inequality can be applied to specific estimation problems. In Subsection \ref{ss:correctedlinreg} the application to corrected linear regression is presented. In Subsection \ref{ss:sparsepca} we show that the sharp oracle inequality also holds for stationary points of sparse PCA. In Subsections \ref{ss:binaryclass} and \ref{ss:robustreg2} we make use of Theorem \ref{thm:oracle} to derive sharp oracle inequalities also for robust regression and binary linear classification. Finally, in Subsection \ref{ss:robustslope} we propose a new estimator ``Robust SLOPE'' and derive a sharp oracle result.

%% file: sharp-oracle-inequality.tex
\section{Sharp oracle inequality}
\label{s:sharporacle}
In this section we mainly discuss the (deterministic) properties of the population version of the general estimation problem. In particular, we first describe the condition on the (population) risk. Then, we specify the kind of regularizers and their characteristics that are covered by our theory. Finally, we state a first general nonrandom sharp oracle inequality.

\subsection{Conditions on the risk}
In order to guarantee a ``sufficient identifiability'' of the parameter $\beta^0$ that is to be estimated, we assume that the risk satisfies a strong convexity condition on the convex set $\mathcal{C}$. It is worth noticing that this is a condition on a theoretical quantity that can be verified under the assumptions on the nonconvex loss in the specific examples.

\begin{condition}[Two point margin condition]
\label{cond:twopointmargin}
There is an increasing strictly convex non-negative function $G$ with $G(0) = 0$ and a semi-norm $\tau$ on $\mathcal{C}$ such that for all $\beta_1, \beta_2 \in \mathcal{C}$
\begin{equation}
R(\beta_1) - R(\beta_2) - \dot{R}(\beta_2)^T (\beta_1 - \beta_2) \geq G(\tau(\beta_1 - \beta_2)).
\end{equation}
\end{condition}

Condition \ref{cond:twopointmargin} says essentially that the curvature of the risk is sufficiently large in a certain neighborhood of $\beta^0$. As will be demonstrated in the sequel of the paper, there are many examples where the loss function is nonconvex with some additional structural assumptions and yet the population risk is ``well-behaved'' on $\mathcal{B}$.

Condition \ref{cond:twopointmargin} is a condition on the \emph{theoretical} risk. In contrast, Restricted Strong Convexity (RSC) that was introduced in \cite{negahban2012restricted} and \cite{agarwal2012fast} combines the curvature \emph{empirical} risk with the penalty. It was originally designed to analyze the properties of \emph{convex} regularized M-estimators. In \cite{loh2011high} and \cite{loh2015regularized} it was further extended to the case of nonconvex M-estimators. \cite{loh2015statistical} introduces the notion of local Restricted Strong Convexity. The latter one can be seen as a two point margin condition on the sample version of the problem on the set $\mathcal{C}$.

\subsection{Conditions on the regularization term}

In the $\ell_1$ world one exploits the property that any vector $\beta \in \mathbb{R}^p$ can be decomposed in an ``active'' and a ``non-active'' part. For a subset $S \subset \{ 1, \dots, p \}$ we define the vector $\beta_S$ such that $\beta_{S, j} = \beta_j \mathbbm{1}_{\left\lbrace j \in S \right\rbrace}$. Then the following decomposition holds:
\begin{equation}
\Vert \beta \Vert_1 = \Vert \beta_S \Vert_1 + \Vert \beta_{S^c} \Vert_1.
\end{equation}
The previous equality is a slight abuse of notation: the vectors $\beta_S$ and $\beta_{S^c}$ lie either in $\mathbb{R}^p$, or $\mathbb{R}^{\vert S \vert}$ and $\mathbb{R}^{p - \vert S \vert}$, respectively. This property is usually named ``decomposability''.

The present framework can be applied to more general norm penalties. In \cite{geer2014weakly} the concept of weak decomposability was introduced. It relaxes decomposability by requiring that for all $\beta \in \mathbb{R}^p$ and certain sets $S$ the sum of certain norms of $\beta_S$ and $\beta_{S^c}$ is always smaller than or equal to $\Omega(\beta)$.

\begin{definition}[Weakly decomposable norm, Definition 4.1 in \cite{geer2014weakly}]
For a subset $S \subset \{1, \dots, p \}$ the norm $\Omega$ is said to be weakly decomposable if there is a norm $\Omega^{S^c}$ on $\mathbb{R}^{p - \vert S \vert }$ such that for all $\beta \in \mathbb{R}^p$
\begin{equation}
\Omega(\beta) \geq \Omega(\beta_S) + \Omega^{S^c} (\beta_{S^c}) =: \underline{\Omega}(\beta).
\end{equation}
\end{definition}

\begin{lemma}
\label{lemma:triangleproperty}
Suppose that the norm $\Omega(\cdot)$ is weakly decomposable for a subset $S \subset \left\lbrace 1, \dots, p \right\rbrace$. Then for all $\beta, \beta' \in \mathbb{R}^p$ 
\begin{equation}
\label{eqn:triangleproperty}
\Omega(\beta) - \Omega(\beta') \leq \Omega(\beta_S' - \beta_S) + \Omega(\beta_{S^c}) - \Omega^{S^c}(\beta_{S^c}').
\end{equation}
\end{lemma}
Equation (\ref{eqn:triangleproperty}) is also named \emph{triangle property}. It imitates the properties of the $\ell_1$-norm.

We insist on the fact that the choice of the regularization term has far-ranging consequences on the properties of the estimator as well as on the techniques that are necessary to analyze the estimator. In \cite{geer2015} the concept of weak decomposability was further extended to other norms. As a consequence, the triangle property can be shown to hold for many more cases. In the present framework however, we sacrifice some generality for a more clear exposition of our results.

\subsection{Effective sparsity}
The choice of the penalization deeply influences the estimation performance of the stationary points. In particular, this affects the estimation error part of the oracle inequality. In order to provide a quantitative description of this effect, we first review some concepts introduced in the rich literature about the Lasso. The concepts developed in the $\ell_1$-norm are paradigmatic of the more general notions.

A well-studied condition on the design in the $\ell_1$-penalized linear regression framework are the $\ell_1$ restricted eigenvalue \cite{bickel2009simultaneous} and the more general compatibility constant \cite{van2007deterministic}. As for the well-known $\ell_1$ framework, we recall the (slightly modified) definition of an $\Omega$-eigenvalue.

\begin{definition}[$\Omega$-eigenvalue, \cite{geer2014weakly}]
Let $S$ be an allowed subset of $\{1, \dots, p \}$ and $L > 0$. The $\Omega$-eigenvalue is defined as
\begin{align}
&\delta_\Omega(\tau, L, S) \nonumber \\
&:= \min \left\lbrace \tau ( \beta_S - \beta_{S^c} ) : \Omega(\beta_S) = 1, \Omega^{S^c}(\beta_{S^c}) \leq L  \right\rbrace,
\end{align}
where $\tau$ is the (semi)-norm from the two point margin condition (Condition \ref{cond:twopointmargin}).
\end{definition}

\begin{definition}[$\Omega$-effective sparsity, \cite{geer2014weakly}]
\label{def:effectivesparstiy}
The $\Omega$-effective sparsity is defined as 
\begin{equation}
\Gamma_\Omega^2(\tau, L, S) = \frac{1}{\delta_\Omega^2(L,S)}.
\end{equation}
\end{definition}

\begin{remark}
Effective sparsity can be interpreted as a measure of how well one can distinguish between the active and non-active parts depending on the specific context of the estimation problem. In fact, one can observe that increasing the stretching factor $L$ reduces the ``distance'' between the sets $\Omega(\beta_S) = 1$ and $\Omega^{S^c}(\beta_{S^c}) \leq L$ (as the size of this set increases). In turn, this means that the effective sparsity becomes larger. In particular, the stretching factor $L$ is shown to depend on the tuning parameter $\lambda$. As the amount of noise increases it is observed that the tuning parameter $\lambda$ increases and therefore also the stretching factor. More noise then translates to less distinguishable active and non-active parts.
\end{remark}

\subsection{Main result}
We denote the oracle by $\beta^\star \in \mathcal{C} \subset \mathbb{R}^p$ and the corresponding ``active'' set will be denoted by $S^\star$. The oracle is a nonrandom vector that might be described as an idealized estimator that has additional structural information about the estimation problem. For instance, the oracle could be a vector that ``knows'' how many non-zero entries the underlying truth has. It then minimizes the upper bound of inequality (\ref{eqn:oracleineq}). In other terms, it optimally trades-off the approximation and estimation errors.

\begin{theorem}
\label{thm:oracle}
Let $\tilde{\beta}$ be a stationary point in the sense of inequality (\ref{eqn:stationary}). Suppose that Condition \ref{cond:twopointmargin} is satisfied. Suppose further that the norm $\Omega(\cdot)$ is weakly decomposable. Let $H$ be the convex conjugate \footnote{The convex conjugate of $G(\cdot)$ is defined as $H(v) = \underset{u \geq 0}{\sup} \left\{u v - G(u) \right\} $ see p. 104 of \cite{rockafellar1970convex}. } of $G$. Let $\lambda_\varepsilon > 0$ and $\lambda_* \geq 0$ such that for all $\beta' \in \mathcal{C}$ and a constant $0 \leq \gamma < 1$
\begin{align}
\label{eqn:randombound}
&\left\vert \left(\dot{R}_n (\beta' ) - \dot{R}(\beta') \right)^T(\beta^\star - \beta') \right\vert \nonumber \\
&\leq \lambda_\varepsilon \underline{\Omega}(\beta' - \beta^\star) + \gamma G(\tau(\beta' - \beta^\star)) +\lambda_*.
\end{align}
Let $\lambda > \lambda_\varepsilon$ and $0\leq \delta < 1$. Define $\underline{\lambda} = \lambda -  \lambda_\varepsilon $, $\overline{\lambda} = \lambda + \lambda_\varepsilon + \delta \underline{\lambda}$, and $L = \overline{\lambda}/((1-\delta) \underline{\lambda})$. Then  we have
\begin{align}
&\delta \underline{\lambda} \underline{\Omega}(\tilde{\beta} - \beta^\star) + R(\tilde{\beta}) \nonumber \\
&\leq R(\beta^\star) + (1-\gamma) H \left( \frac{\overline{\lambda} \Gamma_\Omega(\tau, L, S^\star)}{1- \gamma} \right) + 2 \lambda \Omega(\beta^\star_{S^{\star^c}}) + \lambda_*.
\end{align}
\end{theorem}
The proof of this theorem closely follows the proof of Theorem 7.1 in \cite{geer2015}. The main difference lies in the fact that we do not need convexity of the empirical risk $R_n$. Moreover, we allow for an additional term in the bound for the random part. This is crucial in the examples considered in this paper.
The interpretation of the oracle inequality is that a given estimator achieves a rate of convergence that is almost as good (up to an additional constant term that is typically the risk of the oracle) as if it had background knowledge about the sparsity.

\begin{remark}
Condition (\ref{eqn:randombound}) is a bound for the difference between averages $(\dot{R}_n)$ and means $(\dot{R})$. We refer to it as the `Empirical Process Condition'. Main theme in the applications is to show that this condition holds with high probability, for suitable constants $\lambda_\varepsilon, \lambda_*$ and $\gamma$.
\end{remark}

\begin{remark}
The terminology ``sharp'' is referred to the constant `1' in front of the risk in the upper bound of the inequality below. It also refers to the fact that the upper bound does not involve $R(\beta^0)$.
\end{remark}

\begin{remark}
The noise level $\lambda_\varepsilon$ needs to be chosen depending on the specific structure of the problem. The term $\lambda_*$ is (in an asymptotic sense) of lower order than $\lambda_\varepsilon$. Asymptotically, it does not influence the rates.
\end{remark}

\begin{remark}
The estimation error can be measured in the $\tau$ semi-norm by the two point margin condition or in the $\underline{\Omega}$ norm.
\end{remark}

%% file: applications.tex
\section{Applications to specific estimation problems}
\label{s:applications}
In this section several applications of Theorem \ref{thm:oracle} are presented. The first part is dedicated to the ``usual'' entrywise sparsity where the number of active parameters in the target/truth $\beta^0$ is assumed to be smaller than the problem dimension $p$. In this first part the sparsity inducing norm is taken to be $\Omega(\cdot) = \Vert \cdot \Vert_1$. In the last subsection we introduce a new estimator ``Robust SLOPE'' to demonstrate that our framework can be applied also to different penalizing norms.

\subsection{Corrected linear regression}
\label{ss:correctedlinreg}
In this subsection we closely follow the notation in \cite{loh2011high}. We consider the linear model for $i = 1, \dots, n$:
\begin{equation}
\label{eqn:linmod}
Y_i = X_i \beta^0 + \varepsilon_i,
\end{equation}
where $Y_i \in \mathbb{R}$ is a response variable and $X_i \in \mathbb{R}^{1 \times p}$ are i.i.d. copies of a sub-Gaussian random vector $\tilde{X} \in \mathbb{R}^p$ with unknown positive definite covariance matrix $\Sigma_X$, $\beta^0 \in \mathbb{R}^{p \times 1}$ is unknown and $\varepsilon_1, \dots, \varepsilon_n$ are i.i.d. copies of a sub-Gaussian random variable $\tilde{\varepsilon}$ independent of $\tilde{X}$. We say that a random vector $\tilde{X}$ is
sub-Gaussian if $\sup_{\| \beta \|_2 \le 1} \| \tilde X \beta  \|_{\psi_2} <\infty $ where for a real-valued
random variable $Y$, $\| Y \|_{\psi_2} := \inf \{ c >0: \ \exp [ Y^2 / c^2 ] \leq 1 \}$ is the Orlicz norm for the
function $\psi_2(y) := \exp [y^2]$, $y\ge 0 $.

The matrix $X \in \mathbb{R}^{n \times p}$ with rows $X_i$ may be additionally corrupted by additive noise in which case one would observe
\begin{equation}
Z = X + W.
\end{equation}
The matrix $W$ is independent of $X$ and $\varepsilon := (\varepsilon_1, \dots, \varepsilon_n)^T$. Its rows $W_i$ are assumed to be i.i.d. copies of a sub-Gaussian random vector $\tilde{W}$ with expectation zero and known covariance matrix $\Sigma_W$. Thus, the rows are i.i.d. copies of a random vector $\tilde{Z}$.

The estimator in this case is then given by
\begin{small}
\begin{equation}
\label{eqn:additivenoiseest}
\hat{\beta} = \underset{\beta \in \mathbb{R}^p : \Vert \beta \Vert_1  \leq Q}{\arg \min} \ \left\lbrace \frac{1}{2} \beta^T \left(\frac{Z^TZ}{n} - \Sigma_W \right) \beta - \frac{Y^T Z}{n} \beta + \lambda \Vert \beta \Vert_1 \right\rbrace.
\end{equation}
\end{small}

We assume that $Q \geq \Vert \beta^0 \Vert_1$ so that the vector $\beta^0$ lies within the region over which we compute the estimator.
For ease of notation we define

\begin{equation}
\label{eqn:negdefcov}
\hat{\Gamma}_{\text{add}} := \frac{1}{n} Z^T Z - \Sigma_W \ \text{and} \ \hat{\gamma}_{\text{add}} := \frac{1}{n} Z^T Y.
\end{equation}

The empirical risk is then given by
\begin{equation}
R_n(\beta) = \frac{1}{2} \beta^T \hat{\Gamma}_{\text{add}} \beta - \hat{\gamma}_{\text{add}}^T \beta.
\end{equation}

The first and second derivatives of the empirical risk are given by
\begin{equation}
\dot{R}_n(\beta) = \hat{\Gamma}_{\text{add}} \beta - \hat{\gamma}_{\text{add}}, \quad \ddot{R}_n(\beta) = \hat{\Gamma}_{\text{add}}.
\end{equation}

It can be seen that in a high-dimensional setting ($p > n$) the matrix $\hat{\Gamma}_{\text{add}}$ has negative eigenvalues due to the additional noise. The high-dimensional estimation problem is therefore nonconvex.

On the other hand, the population version of the empirical risk is given by
\begin{equation}
R(\beta) = \mathbb{E} R_n(\beta) = \frac{1}{2} \beta^T \Sigma_X \beta  - \beta^T \Sigma_X \beta^0.
\end{equation}
The first and second derivatives are then given by
\begin{equation}
\dot{R}(\beta) = \Sigma_X \beta - \Sigma_X \beta^0, \quad \ddot{R}(\beta) = \Sigma_X.
\end{equation}

The population version of the estimation is therefore convex. The next lemma shows that the risk is not only convex but even strongly convex.
\begin{lemma}
\label{lemma:twopointcorrreg}
The two point margin condition is satisfied with $G(u) = u^2$ and $\tau(\cdot) = \Vert \Sigma_X^{1/2} (\cdot ) \Vert_2$, where $\Sigma_X^{1/2}$ denotes the square root of $\Sigma_X$.
\end{lemma}

The connection between the penalty and the norm $\tau(\cdot)$ is established in the following lemma that gives an expression for the effective sparsity (Definition \ref{def:effectivesparstiy}).

\begin{lemma}
\label{lemma:effectivesparscorrlin}
For $\tau(\cdot) = \Vert \Sigma_X^{1/2} \left( \cdot \right) \Vert_2$ and $\Omega(\cdot) = \Vert \cdot \Vert_1$ we have for any set $S \subseteq \{ 1, \dots, p \}$ with $s = \vert S \vert$ that
\begin{equation}
\Gamma_{\Vert \cdot \Vert_1} \left(\Vert \Sigma_X^{1/2} (\cdot) \Vert_2, L, S \right) = \sqrt{\frac{s}{\Lambda_{\min}(\Sigma_X)}}.
\end{equation}
\end{lemma}

We now state several lemmas that are used to establish the Empirical Process Condition (\ref{eqn:randombound}). 

\begin{lemma}
\label{lemma:quadraticformcorrlinreg}
Define $\underset{\Vert \beta \Vert_2 \leq 1}{\sup} \ \Vert \tilde{Z} \beta \Vert_{\psi_2} =: C_Z < \infty$. We then have for all $\beta ' \in \mathbb{R}^p$ and all $t > 0$
\begin{small}
\begin{align*}
&\left\vert (\beta' - \beta^\star)^T (\hat{\Gamma}_{\text{\emph{add}}} - \Sigma_X) (\beta' - \beta^\star) \right\vert \\ &\leq 12C_Z^2 \sqrt{\frac{8(t + 2(\log(2p) + 4))}{n}} (\beta' - \beta^\star)^T \Sigma_Z (\beta' - \beta^\star) \\
&+ 12 C_Z^2 \sqrt{\frac{16(\log(2p) + 4)}{n}} \Vert \beta' - \beta^* \Vert_1 \sqrt{(\beta' - \beta^\star)^T \Sigma_Z (\beta' - \beta^\star)} \\
&+12 C_Z^2 \left( \frac{t + 2(\log(2p) + 4)}{n} \right) (\beta' - \beta^\star)^T \Sigma_Z (\beta' - \beta^\star) \\
&+ 12 C_Z^2 \left( \frac{2(\log(2p) + 4) }{n} \right) \Vert \beta' - \beta^\star \Vert_1^2.
\end{align*}
\end{small}
with probability at least $1- \exp(-t)$.
\end{lemma}

The following lemma shows how the quadratic form involving the positive definite matrix $\Sigma_Z$ is related to the (quadratic) margin function.
\begin{lemma}
\label{lemma:covariance}
Define $\Lambda_0 := \left( 1 + \frac{\Lambda_{\max}(\Sigma_W)}{\Lambda_{\min} (\Sigma_X) } \right)$. We have for all $u \in \mathbb{R}^p$
\begin{equation}
u^T \Sigma_Z u \leq \Lambda_0 G(\tau(u)),
\end{equation}
where $\Lambda_{\max}(\Sigma_W)$ and $\Lambda_{\min}(\Sigma_X)$ are the largest and smallest eigenvalues of the matrices $\Sigma_W$ and $\Sigma_X$, respectively.
\end{lemma}

\begin{lemma}
\label{lemma:correctedlinregsubgauss}
Define $\underset{\Vert \beta \Vert_2 \leq 1}{\sup} \ \Vert \tilde{X} \beta \Vert_{\psi_2} =: C_X < \infty$, $\underset{\Vert \beta \Vert_2 \leq 1}{\sup} \ \Vert \tilde{W} \beta \Vert_{\psi_2} =: C_W < \infty$, and $ \Vert \tilde{\varepsilon}_i  \Vert_{\psi_2} =: C_\varepsilon < \infty$ for all $i = 1, \dots, n$, and for $t > 0$
\begin{small}
\begin{align*}
\tilde{\lambda}_\varepsilon(t) = &16 (C_Z^2 \Vert \beta^\star \Vert_2 + C_X^2 \Vert \beta^0 \Vert_2 + C_W C_X \Vert \beta^0 \Vert_2 + C_Z C_\varepsilon) \\
&\cdot \left(2 \sqrt{\frac{2t + \log p}{n}} + \frac{\log p + t}{n} \right).
\end{align*}
\end{small}
Then we have for all $\tilde{\beta} \in \mathbb{R}^p$
\begin{align}
&\left\vert \beta^{\star^T} \left(\hat{\Gamma}_{\text{\emph{add}}} - \Sigma_X \right) (\beta^\star - \tilde{\beta})  +   \left( \hat{\gamma}_{\text{\emph{add}}} - \Sigma_X \beta^0 \right)^T (\beta' - \beta^\star) \right\vert \nonumber \\
&\leq \tilde{\lambda}_\varepsilon(t) \Vert \beta^\star - \tilde{\beta} \Vert_1
\end{align}
with probability at least $1- 4 e^{-t}$.
\end{lemma}

\begin{lemma}
\label{lemma:additivenoiseempirical}
Let $\zeta > 0$ be a constant. Define 
\begin{align*}
\gamma &= 12 C_Z^2 \Lambda_0 \left(  \left( \frac{12 \log (2p) + 16}{n} \right) \zeta^{-1} \left(1 + \zeta \right)  + \zeta \right).
\end{align*}
Then
\begin{small}
\begin{align*}
&\left\vert \left( \dot{R}_n(\tilde{\beta} ) - \dot{R}(\tilde{\beta}) \right)^T(\beta^\star - \tilde{\beta})  \right\vert \\
&\leq \left( \tilde{\lambda}_\varepsilon(\log (2p)) + 24 C_Z^2 \left( \frac{8 (\log (2p) + 4)}{n \zeta} + \frac{2 (\log (2p) + 4)}{n} \right) Q \right) \\
&\cdot \Vert \beta^* - \tilde{\beta} \Vert_1+ \gamma G(\tau(\tilde{\beta} - \beta^*)) \\
\end{align*}
with probability at least $1 - 5 \exp(- \log (2p) )$. If we choose
\begin{align*}
\zeta < (24 C_Z^2)^{-1}  \Lambda_0^{-1} 
\end{align*}
and if we assume that
\begin{align*}
n > 24 C_Z^2  \zeta^{-1} (1+ \zeta) \Lambda_0 ( 12 \log (2p) + 16)
\end{align*}
\end{small}
then $\gamma < 1$. Hence, the Empirical Process Condition (\ref{eqn:randombound}) is satisfied.
\end{lemma}

Combining Lemma \ref{lemma:additivenoiseempirical} with Theorem \ref{thm:oracle} we obtain the following corollary.

\begin{corollary}
Suppose that the assumptions in Lemma \ref{lemma:additivenoiseempirical} hold. Let $\tilde{\beta}$ be a stationary point of the optimization problem (\ref{eqn:additivenoiseest}). Let $\lambda_\varepsilon$ be defined as 

\begin{align*}
\lambda_\varepsilon &= \tilde{\lambda}_\varepsilon(\log (2p)) \\
&\phantom{=}+ 24 C_Z^2 \left( \frac{8 (\log (2p) + 4)}{n \zeta} + \frac{2 (\log (2p) + 4)}{n} \right) Q
\end{align*}

and $\lambda > \lambda_\varepsilon$. Then, we have with probability at least $1 - 5 \exp(-\log(2p))$
\begin{align*}
&\delta \underline{\lambda} \Vert \tilde{\beta} - \beta^\star \Vert_1 + R(\tilde{\beta} ) \nonumber \\
&\leq R(\beta^\star) + \frac{\overline{\lambda}^2 s^{\star} }{4 \Lambda_{\min}(\Sigma_X) (1-\gamma)}+ 2 \lambda \Vert \beta^{\star}_{S^{\star^c}} \Vert_1.
\end{align*}
\end{corollary}

As far as the asymptotics is concerned, we consider the case where the oracle is $\beta^0$ itself. We notice that the choice $Q = o \left( \sqrt{\frac{n}{\log p}} \right)$ leads to
\begin{align}
&\Vert \tilde{\beta} - \beta^0 \Vert_1 \lesssim \sqrt{\frac{\log p}{n}} s_0 \frac{1}{\Lambda_{\min}(\Sigma_X)(1-\gamma)}   \\
&\text{and} \nonumber \\
&\Vert \tilde{\beta} - \beta^0 \Vert_2^2 \lesssim \frac{s_0}{n} \log p \frac{1}{\Lambda_{\min}(\Sigma_X)^2(1-\gamma)}.
\end{align}

We are able to recover the rates obtained also in \cite{loh2011high}. Furthermore, we notice that the rates of convergence depend on the smallest eigenvalue of the true covariance matrix $\Sigma_X$. This is not surprising since the smallest eigenvalue measures the curvature of the population risk. The larger $\Lambda_{\min}(\Sigma_X)$ is, the higher the curvature, and the ``easier'' the estimation problem becomes. As far as estimators leading to conex optimization problem are concerned, \cite{rosenbaum2010sparse} propose and analyze a method for the errors-in-variables model called MU-selector, where MU stands for matrix uncertainty, for a deterministic noise matrix $W$. In  \cite{rosenbaum2013improved} the MU-selector is further improved to allow for random noise in the observations. The estimator is called Compensated MU selector and has a better estimation performance similar to the method that is proposed in \cite{loh2011high} and analyzed in the present paper. Two further estimators leading to convex optimization problems based on an $\ell_1$, $\ell_2$ and $\ell_\infty$ penalties are proposed in \cite{belloni2016ell}. Finally, \cite{belloni2017linear} define an estimator that achieves minimax optimal rates up to a logarithmic term. \cite{datta2017cocolasso} propose another (convex) method called Convex Conditioned Lasso (CoCoLasso) where the negative definite estimate of the covariance matrix (in a high-dimensional setting) such as in (\ref{eqn:negdefcov}) is replaced by a positive semidefinite matrix. In addition to the previously mentioned papers, we also account for the case where the underlying regression function/curve is not necessarily a linear combination of the $s_0$ variables. The importance of the sharp oracle inequalities for the estimator given in equation (\ref{eqn:additivenoiseest}) is to be seen in this additional property rather than in the derivation that bears the dependence on $\Vert \beta^\star \Vert_2$ and $\Vert \beta^0\Vert_2$.

%
%

\subsection{Sparse PCA}
\label{ss:sparsepca}

Principal component analysis is a widely used dimension reduction technique. Its origins go back to \cite{Pearson1901} and \cite{hotelling1933analysis}. Given an $n \times p$ matrix $X \in \mathbb{R}^{n \times p}$ with i.i.d. rows $\{X_i \}_{i = 1}^n$ the aim is to find a one dimensional representation of the data such that the variance explained by this representation is maximized. The empirical covariance matrix is given by $\hat{\Sigma} = X^TX/n$. We write that $\Sigma_X := \mathbb{E} \hat{\Sigma}$. The target $\beta^0 \in \mathbb{R}^p$ is then given by the eigenvector corresponding to the maximal eigenvalue of the covariance matrix $\Sigma_X$.  An estimator for the first principal component is obtained by maximizing the empirical variance with respect to $\beta \in \mathbb{R}^p$:
\begin{equation}
\label{eqn:pca}
\text{maximize} \ \widehat{\text{Var}}(X\beta) = \beta^T \hat{\Sigma} \beta \ \text{subject to} \ \Vert \beta \Vert_2 = 1.
\end{equation}

The solution of the optimization problem (\ref{eqn:pca}) is the eigenvector corresponding to the maximal eigenvalue of the objective function. An equivalent form (after normalization) of the optimization problem (\ref{eqn:pca}) is the following minimization problem where an objective function is minimized with respect to $\beta$:
\begin{equation}
\label{eqn:pca2}
\text{minimize} \ \frac{1}{4} \Vert \hat{\Sigma} - \beta \beta^T \Vert_F^2.
\end{equation}
Both optimization problems (\ref{eqn:pca}) and (\ref{eqn:pca2}) lead to the same solution after normalization. In this case, even if the optimization problem is nonconvex the solution can be easily computed by finding the eigenvector corresponding to the maximal eigenvalue of the sample covariance matrix $\hat{\Sigma}$.

A major drawback of PCA is that the first principal component is typically a linear combination of \emph{all} the variables in the model. In many applications it is however desirable to sacrifice some variance in order to obtain a sparse representation that is easier to interpret. Furthermore, in a high-dimensional setting PCA has been shown to be inconsistent \cite{johnstone2009consistency}. \cite{nadler2008finite} shows that under the spiked covariance model (\cite{johnstone2004sparse}) in a high-dimensional setting the eigenvector corresponding to the largest eigenvalue of $\hat{\Sigma}$ is not able to recover the truth when the gap between the largest eigenvalue of $\Sigma_X$ and the second-largest is ``small''.

 We need to restrict to a neighborhood of one of the global optima in order to assure convexity and uniqueness of the minimum of the risk. Define $\mathcal{B} = \{ \beta \in \mathbb{R}^p : \Vert \beta - \beta^0 \Vert_2 \leq \eta \}$. Let $\beta^\star \in \mathcal{B}$ be the ``oracle'' as given in Section \ref{s:sharporacle}.

We consider the penalized optimization problem
\begin{equation}
\label{eqn:sparsepca}
\hat{\beta} = \underset{\beta \in \mathcal{B}; \Vert \beta \Vert_1 \leq Q}{\arg \min} \ \frac{1}{4} \Vert \hat{\Sigma} - \beta \beta^T \Vert_F^2 + \lambda \Vert \beta \Vert_1,
\end{equation}
where $\lambda > 0$ and $Q > 0$ are tuning parameters.
The risk is given by
\begin{equation}
R(\beta) = \frac{1}{4} \Vert \Sigma_X - \beta \beta^T \Vert_F^2.
\end{equation}
The first derivative of the risk is given by
\begin{equation}
\dot{R}(\beta) = - \Sigma_X \beta + \Vert \beta \Vert_2^2 \beta.
\end{equation}
The second derivative of the risk is given by
\begin{equation}
\ddot{R}(\beta) = - \Sigma_X + 2 \beta \beta^T + \Vert \beta \Vert_2^2 \mathbbm{1}_{p \times p}.
\end{equation}
The (strong) convexity of the risk on the neighborhood $\mathcal{B}$ depends on the ``signal strength''. In this case the latter is given by the largest singular value of the population covariance matrix $\Sigma_X$. The singular value decomposition of $\Sigma_X$ is given by
\begin{equation}
\Sigma_X = P \Phi^2 P^T,
\end{equation}
where $P^TP = PP^T = \mathbbm{1}_{p \times p}$ and $\Phi = \text{diag}(\phi_1, \dots, \phi_p)$ with $\phi_{\max} = \phi_1 \geq \phi_2 \geq \dots \phi_p \geq 0$.
\begin{assumption}
\label{assumption:sparsepca}~
\begin{enumerate}[i)]
\item We assume that the features $X_1, \dots, X_n$ are i.i.d. copies of a sub-Gaussian random vector $\tilde{X} \in \mathbb{R}^{1 \times p}$ with positive definite covariance matrix $\Sigma_X$.
\item It is assumed that for some $\xi > 0$
\begin{equation}
\phi_{\max} \geq \phi_j + \xi, \ \text{for all} \ j \neq 1.
\end{equation}
\item We assume that $\xi > 3 \eta$.
\end{enumerate}
\end{assumption}
\begin{remark}
Assumption \ref{assumption:sparsepca} $ii)$ is often referred to as spikiness condition. It says that the signal should be sufficiently well separated from the other principal components.
\end{remark}
\begin{remark}
What needs to be further explained is the third assumption. In order for the population risk to be convex in the neighborhood $\mathcal{B}$ we require a sufficiently large gap between the largest eigenvalue of the true covariance matrix $\Sigma_X$ and its remaining eigenvalues. One might object that the assumption of starting with a ``good'' starting value is not realistic. However, a consistent initial estimate with a slow rate of convergence is given in \cite{vu2013fantope}.
\end{remark}
The following lemma guarantees that the risk is strictly convex around one of the local minima of the population risk.
\begin{lemma}[Lemma 12.7 in \cite{geer2015}]
\label{lemma:twopointpca}
Suppose that Assumption \ref{assumption:sparsepca} is satisfied. Then for all $\beta \in \mathcal{B}$ we have
\begin{equation}
\Lambda_{\min} (\ddot{R}(\beta) ) \geq 2 \phi_{\max} (\xi - 3\eta),
\end{equation}
where $\Lambda_{\min}(\ddot{R}(\beta))$ is the smallest eigenvalue of the Hessian $\ddot{R}(\beta)$ on the set $\mathcal{B}$.
\end{lemma}
The next lemma shows that the risk is indeed sufficiently convex.
\begin{lemma}
\label{lemma:twopointspca}
Suppose that Assumption \ref{assumption:sparsepca} is satisfied. The two point margin condition is satisfied on $\mathcal{B}$ with $\tau(\cdot) = \Vert \cdot \Vert_2 $ and $G(u) = 2\phi_{\max} (\xi - 3 \eta) u^2$.
\end{lemma}

As we now have a different norm $\tau(\cdot)$ as compared to the sparse corrected linear regression case, we also obtain a different effective sparsity:

\begin{lemma}
\label{lemma:effectivesparsitysparsepca}
For $\tau( \cdot) = \Vert \cdot \Vert_2$ and $\Omega(\cdot) = \Vert \cdot \Vert_1$ we have for any set $S \subseteq \{1, \dots, p \}$ with $s = \vert S \vert$ that
\begin{equation}
\Gamma_{\Vert \cdot \Vert_1}(\Vert \cdot \Vert_2, L, S) = \sqrt{s}.
\end{equation}
\end{lemma}

The following lemma shows that the Empirical Process Condition \ref{eqn:randombound} holds with large probability with appropriate constants.


\begin{lemma}
\label{lemma:spcahighprob}
Define $\underset{\Vert \beta \Vert_2 \leq 1}{\sup} \ \Vert \tilde{X} \beta \Vert_{\psi_2} =: C_X < \infty$ and for $t > 0$
\begin{align*}
\tilde{\lambda}_\varepsilon(t) = 4 C_X^2 ( \Vert \beta^0 \Vert_2 + \eta) \left( 2\sqrt{\frac{2t + \log p}{n}} + \frac{t + \log p}{n} \right).
\end{align*}
Let $\zeta > 0$ be a constant. Then with $\Lambda_1 := 12 C_X^2 \Lambda_{\max}(\Sigma_X) /(\phi_{\max}(\xi - 3\eta))$ and
\begin{align*}
\gamma =& \Lambda_1 \left(  \left( \frac{12 \log (2p) + 16}{n} \right) \zeta^{-1} (1+\zeta) + \zeta  \right)\\
\end{align*}
we have for all $\tilde{\beta} \in \mathcal{B}$
\begin{align}
&\left\vert \left(\dot{R}_n(\tilde{\beta}) - \dot{R}(\tilde{\beta}) \right)^T (\tilde{\beta} - \beta^\star)  \right\vert  \nonumber \\
&\leq \tilde{\lambda}_\varepsilon(\log(2p)) \nonumber  \\
&+  24 C_X^2 Q \frac{16(\log(2p) + 4)}{2n \zeta}   \Vert \tilde{\beta} - \beta^\star \Vert_1 \nonumber \\
&+ 24 C_X^2 Q \frac{2 (\log(2p) + 4)}{n} \Vert \tilde{\beta} - \beta^\star \Vert_1 +  \gamma G(\tau(\tilde{\beta} - \beta^\star))
\end{align}
with probability at least $1- 2 \exp\left(-\log(2p) \right)$. If we choose
\begin{align*}
\zeta < \Lambda_1^{-1} \text{and we assume} \ n >  \zeta^{-1} (1 + \zeta) \Lambda_1 \log p
\end{align*}
we have $\gamma < 1$. Hence, the Empirical Process Condition (\ref{eqn:randombound}) is satisfied.
\end{lemma}

By combining Lemma \ref{lemma:spcahighprob} and Theorem \ref{thm:oracle} we obtain the following corollary.
\begin{corollary}
Let $\tilde{\beta}$ be a stationary point of the optimization problem (\ref{eqn:sparsepca}). Suppose that the conditions of Lemma \ref{lemma:spcahighprob} are satisfied. Let in particular \linebreak $\tilde{\lambda}_\varepsilon(\log(2p))$ be as in Lemma \ref{lemma:spcahighprob}. Define 
\begin{align*}
\lambda_\varepsilon =& \tilde{\lambda}_\varepsilon(\log(2p)) \\
&+ 24 C_X^2 Q \left( \frac{16(\log(2p) + 4)}{2n \zeta} + \frac{2 (\log(2p) + 4)}{n} \right).
\end{align*}

Then we have with probability at least $1- 2 \exp(-\log(2p))$
\begin{align}
&\delta \underline{\lambda} \Vert \tilde{\beta} - \beta^\star \Vert_1 + R(\tilde{\beta}) \nonumber \\
&\leq R(\beta^\star) +\frac{\overline{\lambda}^2 s^\star}{8\phi_{\max}(\xi - 3 \eta)(1-\gamma)}+ 2 \lambda \Vert \beta^{\star}_{S^{\star^c}} \Vert_1.
\end{align}
\end{corollary}

For the asymptotics we assume that $Q  = o \left( \sqrt{\frac{n}{\log p}} \right)$. For simplicity, we take the oracle to be $\beta^0$ itself. Then $\lambda \asymp \sqrt{\log p/ n}$ and
\begin{align}
&\Vert \tilde{\beta} - \beta^0 \Vert_1 \lesssim \sqrt{\frac{\log p}{n}} s_0 \frac{1}{\phi_{\max}(\xi - 3 \eta)(1-\gamma)}  \\
& \text{and} \nonumber \\
&\Vert \tilde{\beta} - \beta^0 \Vert_2^2 \lesssim \frac{s_0}{n} \log p  \frac{1}{\phi_{\max}^2(\xi - 3 \eta)^2(1-\gamma)}.
\end{align}

We see that the rates depend on the gap between the largest eigenvalue of the matrix $\Sigma_X$ and the remaining eigenvalues. It is again not surprising since the estimation problem becomes ``easier'' the larger this gap is.

\subsection{Robust regression}
\label{ss:robustreg2}
We consider the linear model for all $i = 1, \dots, n$ and with $X_i \in \mathbb{R}^{1 \times p}$ i.i.d. copies of a sub-Gaussian random vector $\tilde{X} \in \mathbb{R}^{1 \times p}$ : $\underset{\Vert \beta \Vert_2 \leq 1}{\sup} \ \Vert \tilde{X} \beta \Vert_{\psi_2} =: C_X < \infty$. 
\begin{equation}
Y_i = X_i \beta^0 + \varepsilon_i,
\end{equation}
where we assume that the distribution of the errors is symmetric around $0$. We also assume that the errors $\varepsilon_1, \dots, \varepsilon_n$ are independent of the features $X_1, \dots, X_n$. In case of outliers and heavy-tailed noise in the linear regression model the quadratic loss typically fails due to its unbounded derivative. Alternatives to the quadratic loss are given by e.g. the Cauchy loss.

The empirical risk is given by
\begin{equation}
R_n(\beta) = \frac{1}{n} \sum_{i = 1}^n \rho(Y_i - X_i \beta).
\end{equation}
Its first derivative is given by
\begin{equation}
\dot{R}_n(\beta) = -\frac{1}{n} \sum_{i = 1}^n \dot{\rho} (Y_i - X_i \beta) X_i^T.
\end{equation}
Its second derivative is given by
\begin{equation}
\ddot{R}_n(\beta) = \frac{1}{n} \sum_{i = 1}^n \ddot{\rho}(Y_i - X_i \beta) X_i^T X_i.
\end{equation}

\begin{assumption}~
\label{assumption:robustreg}
\begin{enumerate}[i)]
\item Lipschitz continuity of the loss: there exists $\kappa_1 > 0$ such that
\begin{equation*}
\vert \dot{\rho}(u) \vert \leq \kappa_1, \ \text{for all} \ u \in \mathbb{R}.
\end{equation*}
\item Lipschitz continuity of the first derivative of the loss: there exists $\kappa_2 > 0 $ such that
\begin{equation*}
\vert \ddot{\rho}(u) \vert \leq \kappa_2, \ \text{for all} \ u \in \mathbb{R}.
\end{equation*}
\item Local curvature condition: Define the tail probability as
\begin{equation*}
\varepsilon_T = \mathbb{P} \left( \vert \varepsilon_i \vert \geq \frac{T}{2} \right).
\end{equation*}
It is assumed that for $T > 0$
\begin{equation*}
\frac{7}{2} \alpha_T := \underset{\vert u \vert \leq T}{\min} \ \ddot{\rho}(u) > 0.
\end{equation*}
\end{enumerate}

\end{assumption}
We notice that for our framework we need to assume that also the first derivative of the loss is Lipschitz continuous. In \cite{loh2015statistical} the assumption is weaker in the sense that it is only required that the second derivative of the loss is not ``too negative''.

The usual (typically uncomputable) ``argmin''-type estimator is then given by

\begin{equation}
\label{eqn:robustregest}
\hat{\beta} = \underset{\beta \in \mathcal{B}: \Vert \beta \Vert_1 \leq Q}{\arg \min} \ R_n(\beta) + \lambda \Vert \beta \Vert_1,
\end{equation}
where $\lambda > 0$ and $Q > 0$ are tuning parameters.

We now cite a proposition from \cite{loh2015statistical} that establishes the restricted strong convexity conditions. It shows how the different (tuning) parameters are intertwined.

\begin{proposition}[Adapted from Proposition 2 in \cite{loh2015statistical}]
\label{prop:rsc}
Suppose that $X_1, \dots, X_n$ are i.i.d. copies of a sub-Gaussian random vector $\tilde{X}$ with positive definite covariance matrix $\Sigma_X$. Assume also that
\begin{align}
\label{eqn:robustregpar}
c C_X^2 \left( \varepsilon_T^{1/2} + \exp \left( - \frac{c' T^2}{C_X^2 \eta^2 } \right) \right) \nonumber \\
\leq \frac{\alpha_T}{(7/2) \alpha_T + \kappa_2} \frac{\Lambda_{\min}(\Sigma_X)}{2}.
\end{align}
and that the loss function satisfies Assumption \ref{assumption:robustreg} and that $n \geq c_0 s \log p$. Then we have with probability at least $1- c \exp(-c' \log p)$ for all $\beta_1, \beta_2 \in \mathcal{B}$
\begin{align*}
&\left(\dot{R}_n(\beta_1) - \dot{R}_n(\beta_2) \right)^T \left(\beta_1 - \beta_2 \right) \\
&\geq \alpha \Vert \beta_1 - \beta_2 \Vert_2^2 -  \xi \Vert \beta_1 - \beta_2 \Vert_1^2,
\end{align*}
where
\begin{align*}
\alpha = \frac{7}{2} \alpha_T \frac{\Lambda_{\min}(\Sigma_X)}{16}, \ \text{and} \ \xi =\frac{C ((7/2) \alpha_T +  \kappa_2)^2 C_X^2 T^2 }{\eta^2}.
\end{align*}
\end{proposition}

\begin{remark}
We require a more conservative bound in equation (\ref{eqn:robustregpar}). Instead of a constant $1$ in front of $\alpha_T$ in the denominator of the upper bound we need a constant that is larger than $1$, e.g. $7/2$. This slightly stronger assumption also translates to a requirement on the sample size: the constant $c_0$ here is larger than the constant $c_0$ in \cite{loh2015statistical}. We notice also that the constants $c$ and $c'$ on the left-hand side of inequality (\ref{eqn:robustregpar}) might be different from the one in \cite{loh2015statistical} (see also Lemma \ref{lemma:fourthmomentbound}).
\end{remark}
The following lemma says that the (theoretical) risk of the robust loss functions is strongly convex on $\mathcal{B}$.

\begin{lemma}[Two point margin for nonconvex robust losses]
\label{lemma:twopointrobust}
Suppose that Assumption \ref{assumption:robustreg} is satisfied. The two point margin condition holds with $G(u) = 3 \alpha_T u^2$ and $\tau(\cdot) = \left\Vert \Sigma_X^{1/2} (\cdot) \right\Vert_2$.
\end{lemma}

As far as the effective sparsity is concerned, we notice that the norm $\tau(\cdot)$ and the penalty term $\Vert \cdot \Vert_1$ are the same as in Subsection \ref{ss:correctedlinreg}. Lemma \ref{lemma:effectivesparscorrlin} therefore applies also in this case.

\begin{lemma}
\label{lemma:robustregremp}
Suppose that Assumption \ref{assumption:robustreg} combined with equation (\ref{eqn:robustregpar}) holds. Assume that $n \geq c_0 s \log p$. Define
\begin{equation}
\lambda_\varepsilon = 4 \kappa_1 C_X \sqrt{\frac{2\log (2p)}{n}} + 2 \xi Q \frac{ \log p}{n}.
\end{equation}
With $\gamma = \kappa_2/ (3\alpha_T)$ we then have with probability at least $1- c \exp( - c' \log p)$ that for all $\beta' \in \mathcal{B}$
\begin{align}
\left(\dot{R}_n(\beta') - \dot{R}(\beta') \right)^T(\beta^\star - \beta') \nonumber \\
\leq \lambda_\varepsilon \Vert \beta' - \beta^* \Vert_1 + \gamma G(\tau(\beta' - \beta^\star)).
\end{align}
Assuming that $\frac{\kappa_2}{\alpha_T} <3$ we see that the Empirical Process Condition (\ref{eqn:randombound}) is satisfied.
\end{lemma}

Combining Theorem \ref{thm:oracle} and Lemma \ref{lemma:robustregremp} we have the following corollary.

\begin{corollary}
Let $\tilde{\beta}$ be a stationary point of the objective function (\ref{eqn:robustregest}). Suppose that the conditions of Lemma \ref{lemma:robustregremp} are satisfied.
Then we have with probability at least $1- c \exp( - c' \log p )$
\begin{align}
\delta \underline{\lambda} \Vert \tilde{\beta} - \beta^\star \Vert_1 + R(\tilde{\beta}) &\leq R(\beta^\star) + \frac{\overline{\lambda}^2 s^\star}{12\alpha_T \Lambda_{\min}(\Sigma_X) (1-\gamma)} \nonumber \\
&\phantom{aa} +2 \lambda \Vert \beta^{\star}_{S^{\star^c}} \Vert_1.
\end{align}
\end{corollary}

The asymptotics in this case is as follows: assuming $Q = o \left(\sqrt{\frac{n}{\log p}} \right)$ and therefore a tuning parameter (up to constants) $\lambda \asymp \sqrt{\frac{\log p}{n}}$ we obtain
\begin{align*}
&\Vert \tilde{\beta} - \beta^0 \Vert_1 \lesssim \sqrt{\frac{\log p}{n} } s_0 \frac{1}{\alpha_T \Lambda_{\min}(\Sigma_X)(1- \gamma)} \\
&\text{and} \\
&\Vert \tilde{\beta} - \beta^0 \Vert_2^2 \lesssim \frac{s_0}{n} \log p \frac{1}{\Lambda_{\min}(\Sigma_X)^2 \alpha_T^2(1-\gamma)}.
\end{align*}

\subsection{Binary linear classification}
\label{ss:binaryclass}

In binary linear classification one is interested in estimating the correct group assignment ($1$ or $0$) of the output $Y_i \in \{ 0,1\}$ given the observations $X_i \in \mathbb{R}^p$  for all $i = 1, \dots, n$. The conditional probability is assumed to be given by

\begin{equation}
\mathbb{P} \left(Y_i = 1 \vert X_i = x_i \right) = \sigma(x_i \beta^0),
\end{equation}
where $\beta^0 \in \mathbb{R}^p$ is the quantity that we aim at estimating. The function $\sigma(\cdot): \mathbb{R} \rightarrow (0,1)$ is ``sufficiently regular''.

The empirical risk and population risk are given by
\begin{equation}
R_n(\beta) = \frac{1}{n} \sum_{i = 1}^n \left( Y_i - \sigma(\beta^T X_i^T) \right)^2, \quad R(\beta) = \mathbb{E} R_n(\beta).
\end{equation}
The function $\sigma$ is given by
\begin{equation}
\sigma(z) = \frac{1}{1+e^{-z}} \ \text{for all} \ z \in \mathbb{R}.
\end{equation}
The estimator under study is defined as
\begin{equation}
\label{eqn:binarylinest}
\hat{\beta} = \underset{\beta \in \mathcal{B}}{\arg \min} \ R_n(\beta) + \lambda \Vert \beta \Vert_1,
\end{equation}
where for some constant $\eta > 0$ the neighborhood $\mathcal{B}$ is given by \linebreak $\mathcal{B} = \left\lbrace \beta \in \mathbb{R}^p : \Vert \beta - \beta^0 \Vert_2 \leq \eta \right\rbrace$ and $\lambda > 0$ is a tuning parameter.

\begin{remark}
In order to obtain an initial estimate that is ``sufficiently'' close to the target one may use (under appropriate distributional assumptions) an estimator such as the one proposed in \cite{plan2013one}, \cite{plan2013robust} and \cite{plan2017high}.
\end{remark}

\begin{assumption}~
\label{ass:binlinclass}
\begin{enumerate}[i)]
\item For all $i = 1, \dots, n$ it is assumed that the $X_i$'s are i.i.d. copies of a sub-Gaussian random vector $\tilde{X} \in \mathbb{R}^{p}$ with $C_X := \underset{\Vert \beta \Vert_2 \leq 1}{\sup} \ \Vert \tilde{X} \beta \Vert_{\psi_2} < \infty$.
\item For all $i = 1, \dots, n$ and for some constant $K_2 > 0$ it is assumed that for all $\beta \in \mathcal{B}$: $\vert X_i \beta \vert \leq K_2$.
\end{enumerate}
\end{assumption}

\begin{remark}
Comparable assumptions on the features $X_i$ can be found in \cite{mei2016landscape} and \cite{song2017pulasso} where similar nonconvex estimation problems are discussed and analyzed.
\end{remark}

We now show that the risk is strongly convex on the neighborhood $\mathcal{B}$ of $\beta^0$.

\begin{lemma}
\label{lemma:twopointmarginbinlinclass}
Suppose that Assumption \ref{ass:binlinclass} $i)$ and $ii)$ are satisfied. Define
\begin{align*}
V = \underset{u \in [-K_2, K_2]}{\min} \ \sigma'(u) > 0.
\end{align*}
Assume that $V \Lambda_{\min}(\Sigma_X) > 5 C_X^3 \eta$. Then we have for all $\beta_1, \beta_2 \in \mathcal{B}$
\begin{align*}
&R(\beta_1) - R(\beta_2) - \dot{R}(\beta_2)^T(\beta_1-\beta_2) \\
&\geq 2(V \Lambda_{\min}(\Sigma_X) - 5 C_X^3 \eta) \Vert \beta_1 - \beta_2 \Vert_2^2 
=: G(\tau(\beta_1 - \beta_2)).
\end{align*}
\end{lemma}

We notice that the norm $\tau(\cdot)$ and the penalty are the same as in Subsection \ref{ss:sparsepca}. The effective sparsity is therefore given by Lemma \ref{lemma:effectivesparsitysparsepca}.

The following lemma is used to show that the empirical process part is bounded with high probability.
\begin{lemma}
\label{lemma:empirprocbinlinclass}
Let $\tilde{\varepsilon}_1, \dots, \tilde{\varepsilon}_n$ be i.i.d. Rademacher random variables independent of $\tilde{X}$. Define $C_{\tilde{\varepsilon}} = \Vert \tilde{\varepsilon}_i \Vert_{\psi_2}$ and
\begin{align*}
\lambda_\varepsilon &= 16(6K_2 +2) C_{\tilde{\varepsilon}} C_X \left(\sqrt{\frac{4 \log(p+1)}{n}} + \frac{\log(p+1)}{n} \right) \\
&\phantom{aa}+ C_X \sqrt{\frac{\log p}{n}}
\end{align*}
We have for all $\beta' \in \mathcal{B}$
\begin{align*}
\left\vert \left( \dot{R}_n(\beta') - \dot{R}(\beta') \right)^T (\beta^\star - \beta') \right\vert \leq \lambda_\varepsilon \Vert \beta' - \beta^\star \Vert_1 + \frac{8K_2 \log p}{n}
\end{align*}
with probability at least $1 - (2+ j_0  + \lceil \log_2(2 \sqrt{p} \eta) \rceil ) \exp(-\log(p) ))$, where $j_0$ is the smallest positive integer such that $j_0 + 1 > \log_2 n$.
\end{lemma}

\begin{corollary}
Let $\tilde{\beta}$ be a stationary point of the objective function (\ref{eqn:binarylinest}). Suppose that Assumption \ref{ass:binlinclass} is satisfied. 
Then we have with probability at least $1 - (2+ j_0  + \lceil \log_2(2 \sqrt{p} \eta) \rceil ) \exp(-\log(p) ))$
\begin{align}
&\delta \underline{\lambda}\Vert \tilde{\beta} - \beta^\star \Vert_1 + R(\tilde{\beta}) \nonumber \\
&\leq R(\beta^\star) +\frac{\overline{\lambda}^2 s^\star}{2(V \Lambda_{\min}(\Sigma_X) - 5 C_X^3 \eta) \Lambda_{\min}(\Sigma_X)} \nonumber \\
&\phantom{aa} + 2 \lambda \Vert \beta^{\star}_{S^{\star^c}} \Vert_1 + \frac{8 K_2 \log p}{n}.
\end{align}
\end{corollary}

As far as the asymptotics is concerned, we see that with  $\lambda \asymp \sqrt{\frac{\log p}{n}}$ and taking $\beta^\star = \beta^0$ in the previous corollary we have
\begin{align*}
&\Vert \tilde{\beta} - \beta^0 \Vert_1 \lesssim \sqrt{\frac{\log p}{n} } s_0 \frac{1}{ 2(V \Lambda_{\min}(\Sigma_X) - 5 C_X^3 \eta) \Lambda_{\min}(\Sigma_X) } \\
&\text{and} \\
&\Vert \tilde{\beta} - \beta^0 \Vert_2^2 \lesssim \frac{s_0}{n} \log p \frac{1}{ 4(V \Lambda_{\min}(\Sigma_X) - 5 C_X^3 \eta)^2 \Lambda_{\min}(\Sigma_X)^2 }.
\end{align*}

\subsection{Robust SLOPE}
\label{ss:robustslope}

As an example for a nonconvex M-estimator that is used with a penalty that is not the $\ell_1$-norm, we consider a robust version (i.e. using a robust loss function as in Subsection \ref{ss:robustreg2} instead of the quadratic loss) of the estimator proposed in \cite{bogdan2013statistical}.

Let $\lambda_1 \geq \lambda_2 \geq \dots \geq \lambda_p \geq 0$ and $\lambda_1 > 0$. For $\beta \in \mathbb{R}^p$ the sequence $\vert \beta \vert_{(1)} \geq \vert \beta \vert_{(2)} \geq \dots \geq \vert \beta \vert_{(p)}$ represents the absolute values of the entries of $\beta$ in increasing order. The sorted $\ell_1$-norm is then defined as 
\begin{equation}
J_\lambda(\beta) = \lambda_1 \vert \beta \vert_{(1)} + \lambda _2 \vert \beta \vert_{(2)} + \dots + \lambda_p \vert \beta \vert_{(p)} = \sum_{j = 1}^p \lambda_j \vert \beta \vert_{(j)}.
\end{equation}

We define the robust SLOPE estimator as
\begin{equation}
\label{eqn:robuslope}
\hat{\beta} = \underset{\beta \in \mathcal{B}: \Vert \beta \Vert_1 \leq Q}{\arg \min} \ R_n(\beta) + \mu J_\lambda(\beta),
\end{equation}
where $\mu > 0$ and $Q > 0$ are tuning parameters.

\begin{lemma}[Lemma 6.13 in \cite{geer2015} and Lemma 15 in \cite{stucky2017sharp}]
\label{lemma:slopeweakdecom}
The sorted $\ell_1$-norm is weakly decomposable with
\begin{equation}
\Omega^{S^c}(\beta_{S^c}) = \sum_{l = 1}^r \lambda_{p - r+l} \vert \beta \vert_{(l, S^c)},
\end{equation}
where $r = p -s$ and $\vert \beta \vert_{(1, S^c)} \geq \dots \geq \vert \beta \vert_{(r, S^c)}$ is the sequence of ordered absolute values indexed in the set $S$. The norm $\underline{\Omega}(\cdot)$ is defined as $ \underline{\Omega}(\beta) := J_\lambda(\beta_S) + \Omega^{S^c} (\beta_{S^c})$.
\end{lemma}

The following lemma, which is in part given also in \cite{stucky2017sharp} after the definition of the square root SLOPE, allows us to show that the Empirical Process Condition \ref{eqn:randombound} is satisfied with high probability for $\underline{\Omega}(\cdot)$ given in Lemma \ref{lemma:slopeweakdecom}.

\begin{lemma}
\label{lemma:slopel1}
Suppose that $\lambda_p > 0$. For all $\beta \in \mathbb{R}^p$ we have that
\begin{equation}
\Omega(\beta) \geq \underline{\Omega}(\beta) \geq \lambda_p \Vert \beta \Vert_1
\end{equation}
and consequently for the dual norm of $\Omega(\cdot)$ we have for all $w \in \mathbb{R}^p$
\begin{equation}
\Omega_*(w) \leq \lambda_p \Vert w \Vert_\infty.
\end{equation}
\end{lemma}

The effective sparsity of the sorted $\ell_1$-norm is given in the following lemma.

\begin{lemma}
\label{lemma:effectivesparsityslope}
For $\tau(\cdot) = \Vert \Sigma_X^{1/2} (\cdot) \Vert_2$ and $\Omega(\cdot) = J_\lambda(\cdot)$ we have for any set $S \subseteq \{ 1, \dots, p \}$ with $s = \vert S \vert$ that
\begin{equation}
\Gamma_{J_\lambda} \left(\Vert \Sigma_X^{1/2} (\cdot) \Vert_2, L, S \right) = \lambda_1 \sqrt{\frac{s}{\Lambda_{\min}(\Sigma_X)}}.
\end{equation}
\end{lemma}

\begin{corollary}
Suppose that Assumption \ref{assumption:robustreg} is satisfied. Let $\tilde{\beta}$ be a stationary point of the objective function (\ref{eqn:binarylinest}). Then we have with probability at least $1- c \exp( - c' \log p )$
\begin{align}
&\delta \underline{\lambda} \underline{\Omega}(\tilde{\beta} - \beta^\star) + R(\tilde{\beta}) \nonumber \\
 &\leq R(\beta^\star) +\frac{\overline{\mu}^2 \lambda_1^2 s^\star}{12\Lambda_{\min}(\Sigma_X)\alpha_T (1-\gamma)} + 2 \mu \Omega( \beta^{\star -}).
\end{align}
\end{corollary}


%% file: discussion2.tex
\section{Discussion}
We have extended the general framework to derive sharp oracle inequalities in Chapter 7 of \cite{geer2015} for convex optimization problems to \emph{nonconvex} optimization problems. Stationary points of certain nonconvex regularized M-estimators are shown to satisfy sparsity oracle inequalities provided that the risk satisfies a (restricted) form of strong convexity. In addition, we have demonstrated that our framework can be applied to weakly decomposable norms. So far, we have restricted ourselves to norm penalized estimators since the techniques used to bound the empirical processes rely on the properties of norms. The derived oracle inequalities are sharp in the sense that they reveal closeness to the best approximation in the model class plus a remainder term which can be seen as the estimation error. These sharp oracle inequalities extend the rates of convergence obtained in previous work.

%% file: lemmasectionoracle.tex
\section{Proof of the lemma in Section \ref{s:sharporacle}}
\begin{proof}[Proof of Lemma \ref{lemma:triangleproperty}]
By the weak decomposability we have
\begin{align*}
\Omega(\beta') \geq \Omega(\beta'_{S}) + \Omega^{S^c}(\beta'_{S^c})
\end{align*}
which is equivalent to
\begin{align*}
- \Omega(\beta') \leq - \Omega(\beta'_{S}) - \Omega^{S^c}(\beta'_{S^c}).
\end{align*}
By the above and the triangle inequality we have
\begin{align*}
\Omega(\beta) - \Omega(\beta') &= \Omega(\beta_S + \beta_{S^c}) - \Omega(\beta') \\
&\leq \Omega(\beta_S + \beta_{S^c}) - \Omega(\tilde{\beta}_S) - \Omega^{S^c}(\tilde{\beta}_{S^c}) \\
&\leq \Omega(\beta_S - \tilde{\beta}_S) + \Omega(\beta_{S^c}) - \Omega^{S^c}(\tilde{\beta}_{S^c}).
\end{align*}
\end{proof}

%% file: proof-maintheorem.tex
\section{Proof of Theorem \ref{thm:oracle}}
\begin{proof}[Proof of Theorem \ref{thm:oracle}]
As remarked earlier the proof of this theorem closely follows the proof of Theorem 7.1 in \cite{geer2015} but there are some important differences. We start by considering a Taylor expansion of the risk around a stationary point $\tilde{\beta}$ of the objective function $R_n(\cdot) + \lambda \Omega(\cdot)$:
\begin{equation*}
R(\beta) = R(\tilde{\beta}) + \dot{R}(\tilde{\beta})^T(\beta - \tilde{\beta}) + \text{Rem}(\tilde{\beta}, \beta),
\end{equation*}
where $\text{Rem}(\tilde{\beta}, \beta)$ is defined as
\begin{equation*}
\text{Rem}(\tilde{\beta}, \beta) = R(\beta) - R(\tilde{\beta}) - \dot{R}(\tilde{\beta})^T (\beta - \tilde{\beta}).
\end{equation*}
Then,
\begin{equation*}
R(\tilde{\beta}) - R(\beta) + \text{Rem}(\tilde{\beta}, \beta) = - \dot{R}(\tilde{\beta})^T(\beta - \tilde{\beta}).
\end{equation*}

\underline{Case 1} Suppose that
\begin{equation*}
\dot{R}(\tilde{\beta})^T(\beta - \tilde{\beta}) \geq \delta \underline{\lambda} \underline{\Omega}( \tilde{\beta} - \beta ) - 2 \lambda \Omega( \beta_{S^c} ) - \lambda_* - \gamma G(\tau(\tilde{\beta} - \beta)).
\end{equation*}
Then we have
\begin{align*}
&\delta \underline{\lambda} \underline{\Omega}( \tilde{\beta} - \beta ) +  R(\tilde{\beta}) + \text{Rem}(\tilde{\beta}, \beta) \\
&\leq  R(\beta) + 2 \lambda \Omega( \beta_{S^c} ) + \lambda_* + \gamma G(\tau(\tilde{\beta} - \beta)) \\
\end{align*}
Hence, as $\text{Rem}(\tilde{\beta}, \beta) \geq 0$ and $\gamma < 1$ we have
\begin{align*}
&\delta \underline{\lambda} \underline{\Omega}( \tilde{\beta} - \beta ) + R(\tilde{\beta})  +  \gamma \text{Rem}(\tilde{\beta}, \beta) \\
&\leq R(\beta) + \lambda_* + \gamma \underbrace{G(\tau(\tilde{\beta} - \beta))}_{\leq \text{Rem}(\tilde{\beta}, \beta)} + 2 \lambda \Omega(\beta_{S^c})
\end{align*}
and therefore
\begin{align*}
\delta \underline{\lambda} \underline{\Omega}( \tilde{\beta} - \beta ) + R(\tilde{\beta} ) \leq R(\beta) + 2 \lambda \Omega( \beta_{S^c} ) + \lambda_*.
\end{align*}
\underline{Case 2} Suppose now that $ \dot{R}(\tilde{\beta})^T(\beta - \tilde{\beta}) \leq \delta \underline{\lambda} \underline{\Omega}( \tilde{\beta} - \beta ) - 2 \lambda \Omega( \beta_{S^c} ) - \lambda_* - \gamma G(\tau(\tilde{\beta} - \beta))$.
By the stationarity of $\tilde{\beta}$ we have that
\begin{align*}
\dot{R}_n(\tilde{\beta})^T(\beta - \tilde{\beta}) + \lambda \tilde{z}^T (\beta - \tilde{\beta} ) \geq 0 \  \forall \beta \in \mathcal{C}.
\end{align*}
This implies 
\begin{align*}
- \dot{R}_n(\tilde{\beta})^T(\beta - \tilde{\beta}) &\leq \lambda \tilde{z}^T \beta - \lambda \Omega ( \tilde{\beta} ) \\
&\leq \lambda \underbrace{\Omega_*( \tilde{z} )}_{\leq 1} \Omega( \beta ) - \lambda \Omega ( \tilde{\beta} ) \\
&\leq \lambda \Omega( \beta) - \lambda \Omega( \tilde{\beta}).
\end{align*}
Or equivalently,
\begin{align*}
0 \leq \dot{R}_n(\tilde{\beta})^T (\beta - \tilde{\beta}) + \lambda \Omega( \beta) - \lambda \Omega( \tilde{\beta}).
\end{align*}
Therefore,
\begin{align*}
&-\dot{R}(\tilde{\beta})^T(\beta- \tilde{\beta})+ \delta \underline{\lambda} \underline{\Omega} (\tilde{\beta} - \beta) \\
&\leq \left(\dot{R}_n(\tilde{\beta}) - \dot{R}(\tilde{\beta}) \right)^T(\beta - \tilde{\beta}) + \lambda \Omega( \beta ) - \lambda \Omega ( \tilde{\beta} ) + \delta \underline{\lambda} \underline{\Omega} (\tilde{\beta} - \beta) \\
&\leq \lambda_\varepsilon \underline{\Omega}( \tilde{\beta} - \beta ) + \lambda \Omega( \beta) - \lambda \Omega( \tilde{\beta}) + \lambda_* + \gamma G(\tau(\tilde{\beta} - \beta)) + \delta \underline{\lambda} \underline{\Omega} (\tilde{\beta} - \beta).
\end{align*}
Hence,
\begin{align*}
&-\dot{R}(\tilde{\beta})^T(\beta - \tilde{\beta}) +\delta \underline{\lambda} \underline{\Omega} (\tilde{\beta} - \beta) \\
&\leq \lambda_\varepsilon (\Omega( \tilde{\beta}_{S} - \beta_{S} ) + \Omega^{S^c}(\tilde{\beta}_{S^c} - \beta_{S^c} )) + \lambda \Omega( \beta) - \lambda \Omega( \tilde{\beta}) \\
&+ \lambda_* + \gamma G(\tau(\tilde{\beta} - \beta)) + \delta \underline{\lambda} \underline{\Omega} (\tilde{\beta} - \beta) \\
&\leq \lambda_\varepsilon (\Omega( \tilde{\beta}_{S} - \beta_{S} ) + \Omega^{S^c}(\tilde{\beta}_{S^c} - \beta_{S^c} ))  \\
&+\lambda \Omega(\tilde{\beta}_S - \beta_S) + \lambda \Omega(\beta_{S^c}) - \lambda \Omega^{S^c} (\tilde{\beta}_{S^c})  \ \text{by the triangle property, cf. Lemma \ref{lemma:triangleproperty}} \\
&+\lambda_* + \gamma G(\tau(\tilde{\beta} - \beta)) + \delta \underline{\lambda} \underline{\Omega} (\tilde{\beta} - \beta) \\
&\leq \lambda_\varepsilon (\Omega( \tilde{\beta}_{S} - \beta_{S} ) + \Omega^{S^c}(\tilde{\beta}_{S^c} - \beta_{S^c} ))  \\
&+\lambda \Omega(\tilde{\beta}_S - \beta_S) + \lambda \Omega(\beta_{S^c}) - \lambda \Omega^{S^c}(\tilde{\beta}_{S^c} - \beta_{S^c}) + \lambda \Omega^{S^c} (\beta_{S^c})  \ \text{by the triangle inequality} \\
&+\lambda_* + \gamma G(\tau(\tilde{\beta} - \beta)) + \delta \underline{\lambda} \underline{\Omega} (\tilde{\beta} - \beta)\\
&\leq \lambda_\varepsilon (\Omega( \tilde{\beta}_{S} - \beta_{S} ) + \Omega^{S^c}(\tilde{\beta}_{S^c} - \beta_{S^c} ))  \\
&+\lambda \Omega(\tilde{\beta}_S - \beta_S) - \lambda \Omega^{S^c}(\tilde{\beta}_{S^c} - \beta_{S^c})  + 2\lambda \Omega(\beta_{S^c})  \ \text{since} \  \Omega^{S^c}(\beta_{S^c}) \leq \Omega(\beta_{S^c}) \\
&+\lambda_* + \gamma G(\tau(\tilde{\beta} - \beta)) + \delta \underline{\lambda} \underline{\Omega} (\tilde{\beta} - \beta).
\end{align*}

Summarizing, we have
\begin{align*}
&-\dot{R}(\tilde{\beta})^T (\beta - \tilde{\beta}) + \delta \underline{\lambda} \underline{\Omega}(\tilde{\beta} - \beta) \\
&\leq (\lambda_\varepsilon + \lambda + \delta \underline{\lambda}) \Omega(\tilde{\beta}_S - \beta_S)  -(1 - \delta) \underline{\lambda}\Omega^{S^c}(\tilde{\beta}_{S^c} - \beta_{S^c}) \\
&+ \lambda_* + \gamma G(\tau(\tilde{\beta} - \beta)) + 2 \lambda \Omega(\beta_{S^c}),
\end{align*}
where we have used that $\underline{\lambda} = \lambda - \lambda_\varepsilon$.
Using that

\begin{align*}
-\dot{R}(\tilde{\beta})^T (\beta - \tilde{\beta}) + \delta \underline{\lambda} \underline{\Omega} (\tilde{\beta} - \beta)\\
\geq 2 \lambda \Omega(\beta_{S^c}) + \lambda_* + \gamma G(\tau(\tilde{\beta} - \beta))
\end{align*}
we obtain
\begin{align*}
0 \leq \overline{\lambda} \Omega(\tilde{\beta}_S - \beta_S) - (1-\delta) \underline{\lambda} \Omega^{S^c} (\tilde{\beta}_{S^c} - \beta_{S^c}).
\end{align*}
Hence,
\begin{equation}
(1- \delta) \underline{\lambda} \Omega^{S^c}( \tilde{\beta}_{S^c} - \beta_{S^c}) \leq \overline{\lambda} \Omega( \tilde{\beta}_S - \beta_S).
\end{equation}
Or equivalently,
\begin{equation}
\Omega^{S^c}( \tilde{\beta}_{S^c} - \beta_{S^c}) \leq \frac{\overline{\lambda}}{(1-\delta) \underline{\lambda}} \Omega( \tilde{\beta}_S - \beta_S) = L \Omega( \tilde{\beta}_S - \beta_S).
\end{equation}
By the effective sparsity we have that
\begin{equation}
\Omega( \tilde{\beta}_S - \beta_S) \leq \tau(\tilde{\beta}_S - \beta_S) \Gamma_\Omega(L, \beta_S, \tau).
\end{equation}
We then have
\begin{align*}
&- \dot{R}(\tilde{\beta})^T (\beta - \tilde{\beta}) + \delta \underline{\lambda} \underline{\Omega}(\tilde{\beta} - \beta) \\
&\leq \overline{\lambda} \tau(\tilde{\beta} - \beta) \Gamma_{\Omega}(L, \beta_S, \tau) + 2 \lambda \Omega( \beta_{S^c}) + \lambda_* + \gamma G(\tau(\tilde{\beta} - \beta)).
\end{align*}
Using Fenchel's inequality, the convexity of $G$, and the two point margin we conclude that
\begin{align*}
&(R(\tilde{\beta}) - R(\beta) + \text{Rem}(\tilde{\beta}, \beta) ) +\delta \underline{\lambda} \underline{\Omega}(\tilde{\beta} - \beta) \\
&\leq (1-\gamma) \tau(\tilde{\beta} - \beta) \frac{\overline{\lambda} \Gamma_{\Omega}(L, \beta_S, \tau)}{1 - \gamma} + 2 \lambda \Omega(\beta_{S^c}) + \lambda_* + \gamma G(\tau(\tilde{\beta} - \beta))  \\
&\leq (1-\gamma) H \left( \frac{\overline{\lambda} \Gamma_{\Omega}(L, \tau, \beta_S)}{1 - \gamma} \right) + (1-\gamma) G(\tau(\tilde{\beta} - \beta)) + 2 \lambda \Omega(\beta_{-S}) \\
&\phantom{\leq} + \lambda_* + \gamma G(\tau(\tilde{\beta} - \beta)) \\
&\leq (1-\gamma) H \left( \frac{\overline{\lambda} \Gamma_{\Omega}(L, \tau, \beta_S)}{1 - \gamma} \right) + \text{Rem}(\tilde{\beta}, \beta) + 2 \lambda \Omega( \beta_{S^c}) +\lambda_* \\
&= (1-\gamma) H \left( \frac{\overline{\lambda} \Gamma_{\Omega}(L, \tau, \beta_S)}{1 - \gamma} \right) + 2 \lambda \Omega( \beta_{S^c}) + \text{Rem}(\tilde{\beta}, \beta)  + \lambda_*.
\end{align*}

Hence,
\begin{align*}
\delta \underline{\lambda} \underline{ \Omega}( \tilde{\beta} - \beta) + R(\tilde{\beta})
\leq R(\beta) + (1-\gamma) H \left( \frac{\overline{\lambda} \Gamma_{\Omega}(L, \tau, \beta_S)}{1 - \gamma} \right) + 2 \lambda \Omega( \beta_{S^c} )  + \lambda_*.
\end{align*}
\end{proof}

%% file: subgaussian.tex
\section{Properties of sub-Gaussian and sub-exponential random variables}

In this section, we summarize and prove some useful facts about sub-Gaussian and sub-exponential random variables. The characterization via the Orlicz norms is used. We refer the reader to Section 2.2 of \cite{van1996weak} for a detailed treatment. For $k = 1,2$ and $x \in \mathbb{R}_{\geq 0}$ we define the function
\begin{equation}
\psi_k(x) = \exp\left(x^k\right) - 1.
\end{equation}
\begin{definition}
The Orlicz norm $\Vert X \Vert_{\psi_k}$ of a random variable $X$ is defined as
\begin{equation}
\label{eqn:orlicznorm}
\Vert X \Vert_{\psi_k} = \inf \left\lbrace C > 0 : \mathbb{E} \psi_k \left( \frac{\vert X \vert}{C} \right) \leq 1 \right\rbrace.
\end{equation}
\end{definition}

\begin{lemma}[Exercise 7 in \cite{van1996weak}]
The infimum in Equation \ref{eqn:orlicznorm} is attained. 
\end{lemma}

Among other properties of a random variable, Orlicz norms describe the behavior of its moments.

\begin{definition}
A random variable $X$ is said to be sub-exponential if $\Vert X \Vert_{\psi_1}  < \infty$.
\end{definition}

\begin{definition}
A random variable $X$ is said to be sub-Gaussian if $\Vert X \Vert_{\psi_2} < \infty $.
\end{definition}
The next lemma gives a bound on the second moment of a sub-Gaussian random variable.
\begin{lemma}
Suppose that $\Vert X \Vert_{\psi_2} < \infty$. Then
\begin{equation}
\mathbb{E} \left[ X^2 \right] \leq 2 \Vert X \Vert_{\psi_2}^2.
\end{equation}
\end{lemma}
\begin{proof}
We have
\begin{align*}
\mathbb{E} \left[ X^2 \right] &= \int_0^\infty \mathbb{P} \left( X^2 > t \right) dt \\
&= \int_0^\infty \mathbb{P} \left( \frac{X^2}{\Vert X \Vert_{\psi_2}^2 } > \frac{t}{\Vert X \Vert_{\psi_2}^2 } \right) dt \\
&= \Vert X \Vert_{\psi_2}^2 \int_0^\infty \mathbb{P} \left( \frac{X^2}{\Vert X \Vert_{\psi_2}^2} > x \right) dx \\
&\leq \Vert X \Vert_{\psi_2}^2 \int_0^\infty \mathbb{E} \left[\exp \left( \frac{X^2}{\Vert X \Vert_{\psi_2}^2 } \right) \right] \exp(-x) dx \\
&\leq 2 \Vert X \Vert_{\psi_2}^2 \int_0^\infty e^{-x} dx\\
&= 2 \Vert X \Vert_{\psi_2}^2.
\end{align*}
\end{proof}

The following lemma gives a bound on the fourth moment of a sub-Gaussian random variable. It allows us to carry over the proof of Lemma 7 in \cite{loh2015statistical} that is used to establish Proposition \ref{prop:rsc}.

\begin{lemma}
\label{lemma:fourthmomentbound}
Suppose that $\Vert X \Vert_{\psi_2} < \infty$. Then
\begin{equation}
\mathbb{E} \left[ X^4 \right] \leq 4 \Vert X \Vert_{\Psi_2}^4.
\end{equation}
\end{lemma}

\begin{proof}
Since $X^4 \geq 0$ we have by the tail summation property
\begin{align*}
\mathbb{E} \left[X^4 \right] &= \int_0^\infty \mathbb{P} \left(X^4 > t \right) dt \\
&= \int_0^\infty \mathbb{P} \left( \frac{X^4}{\Vert X \Vert_{\psi_2}^4} > \frac{t}{\Vert X \Vert_{\psi_2}^4} \right) dt \\
&= 2 \Vert X \Vert_{\psi_2}^4 \int_0^\infty \mathbb{P} \left( \frac{X^2}{\Vert X \Vert_{\psi_2}^2} > x \right) x \ dx, \ \text{by a change of variables} \\
&= 2 \Vert X \Vert_{\psi_2}^4 \int_0^\infty \mathbb{P} \left( \exp \left( \frac{X^2}{\Vert X \Vert_{\psi_2}^2} \right)  > \exp(x)   \right) dx \\
& \leq 2 \Vert X \Vert_{\psi_2}^4 \int_0^\infty \mathbb{E} \left[ \exp \left( \frac{X^2}{\Vert X \Vert_{\psi_2}^2} \right) \right] \exp(-x) x \ dx, \ \text{by Markov's inequality}  \\
&\leq 4 \Vert X \Vert_{\psi_2}^4 \int_0^\infty e^{-x} x \ dx, \ \text{since} \ \mathbb{E} \left[\exp(X^2/ \Vert X \Vert_{\psi_2}^2) \right] \leq 2 \\
&= 4 \Vert X \Vert_{\psi_2}^4 \left( - e^{-x} x \vert_0^\infty + \int_0^\infty e^{-x} dx \right) \\
&= 4\Vert X \Vert_{\psi_2}^4.
\end{align*}
\end{proof}

In the proofs of our results we often need to bound the average of the products of sub-Gaussian random variables. The following lemma says that the product of two sub-Gaussian random variables is sub-exponential.

\begin{lemma}
\label{lemma:productofsubgauss}
Suppose that $\Vert X \Vert_{\psi_2} < \infty$ and that $\Vert Y \Vert_{\psi_2} < \infty$. Then $\Vert X Y \Vert_{\psi_1} \leq \Vert X \Vert_{\psi_2} \Vert Y \Vert_{\psi_2}$.
\end{lemma}

\begin{proof}
We have that
\begin{align*}
&\mathbb{E} \left[ \exp \left( \frac{\vert X Y \vert}{\Vert X \Vert_{\psi_2} \Vert Y \Vert_{\psi_2}} \right) \right] \\
&\leq \mathbb{E} \left[ \exp \left( \frac{X^2}{2 \Vert X \Vert_{\psi_2}^2 } + \frac{Y^2}{2 \Vert Y \Vert_{\psi_2}^2 } \right) \right], \ \text{by Young's inequality} \\
&\leq \mathbb{E} \left[ \frac{1}{2} \exp \left( \frac{X^2}{ \Vert X \Vert_{\psi_2}^2 } \right) + \frac{1}{2}\exp \left( \frac{Y^2}{\Vert Y \Vert_{\psi_2}^2} \right) \right], \ \text{by Young's inequality} \\
&\leq 2.
\end{align*}
\end{proof}

The concentration behavior of the average of products of sub-Gaussian random variables is given in the following lemma.

\begin{lemma}
Let $(X_1, Y_1) , \dots, (X_n, Y_n)$ be i.i.d. copies of $(X,Y)$, where $\Vert X \Vert_{\psi_2} < \infty$ and $\Vert Y \Vert_{\Psi_2} < \infty$. Then for all $t > 0$
\begin{equation}
\mathbb{P} \left( \left\vert  \frac{1}{n} \sum_{i = 1}^n X_i Y_i - \mathbb{E} X Y \right\vert \geq 8 \Vert X \Vert_{\psi_2} \Vert Y \Vert_{\psi_2} \sqrt{\frac{2t}{n}} + 4 \Vert X \Vert_{\psi_2} \Vert Y \Vert_{\psi_2} \frac{t}{n} \right) \leq \exp(-t).
\end{equation}
If $X$ and $Y$ are independent and $\mathbb{E} X = \mathbb{E} Y = 0$, for all $t > 0$
\begin{equation}
\mathbb{P} \left( \left\vert \frac{1}{n} \sum_{i = 1}^n X_i Y_i \right\vert \geq 2 \Vert X \Vert_{\psi_2} \Vert Y \Vert_{\psi_2} \sqrt{\frac{2t}{n}} + \Vert X \Vert_{\psi_2} \Vert Y \Vert_{\psi_2} \frac{t}{n} \right) \leq \exp(-t).
\end{equation}

\end{lemma}

%% file: empirical-process-bounds.tex
\section{Proof of the lemmas in Section \ref{s:applications}}

\subsection{High probability bounds on random quadratic forms}
In the following we provide similar bounds on random quadratic forms as in the Supplement of \cite{loh2011high}. We provide a full proof so that it can be traced down where our explicit constants come from.  
\begin{lemma}
Consider a row-vector $X \in \mathbb{R}^s$ with $\Sigma_X := \mathbb{E} X^TX$. Let $\Sigma_X = I$ and $\underset{\Vert \beta \Vert_2 \leq 1}{\sup} \ \Vert X \beta \Vert_{\psi_2} =: C < \infty$. Then for all $t> 0$, with probability at least $1- \exp(-t)$, it holds that
\begin{equation}
\underset{\Vert u \Vert_2 = 1, \Vert v \Vert_2 = 1}{\sup} \ \left\vert u^T (\hat{\Sigma} - I) v \right\vert \leq 4 C^2 \sqrt{\frac{8(t+\log 2+ 4s)}{n}} + 4 C^2 \left( \frac{t + \log 2 + 4s}{n} \right).
\end{equation}
\end{lemma}

\begin{proof}
Define
\begin{equation*}
\hat{A} := \underset{\Vert u \Vert_2 = 1, \Vert v \Vert_2 = 1}{\sup} \ \left\vert u^T(\hat{\Sigma} - I) v \right\vert.
\end{equation*}
Let $u, v, \tilde{u}$ and $\tilde{v}$ be arbitrary. Then
\begin{align*}
u^T(\hat{\Sigma} - I) v &= \tilde{u}^T(\hat{\Sigma} - I)\tilde{v} \\
&+ (u - \tilde{u})^T(\hat{\Sigma} - I)(v - \tilde{v}) + u^T(\hat{\Sigma} - I)(v - \tilde{v}) +  (u - \tilde{u})^T(\hat{\Sigma} - I) v.
\end{align*}
Thus for all $\varepsilon > 0$ and for $\Vert u - \tilde{u} \Vert_2 \leq \varepsilon$ and $\Vert v - \tilde{v} \Vert_2 \leq \varepsilon$
\begin{equation}
\left\vert u^T(\hat{\Sigma} - I) v \right\vert \leq \left\vert \tilde{u}^T (\hat{\Sigma} - I) \tilde{v} \right\vert + 2 \varepsilon \hat{A} + \varepsilon^2 \hat{A}.
\end{equation}
We now take $\mathcal{S}_\varepsilon$ to be a minimal $\varepsilon$-covering of the unit sphere $\left\lbrace w \in \mathbb{R}^s : \Vert w \Vert_2 = 1 \right\rbrace$. Then $\vert \mathcal{S}_\varepsilon \vert \leq (1 + 2/\varepsilon)^s$ by Lemma 14.27 in \cite{geer2015}. It follows that
\begin{equation}
(1-2\varepsilon - \varepsilon^2) \hat{A} \leq \underset{\tilde{u} \in \mathcal{S}_\varepsilon, \tilde{v} \in \mathcal{S}_\varepsilon}{\max} \ \left\vert \tilde{u}^T (\hat{\Sigma} - I ) \tilde{v} \right\vert.
\end{equation}
For each $\tilde{u}$ and $\tilde{v}$ in the unit sphere, we know that $\Vert X \tilde{u} \Vert_{\psi_2} \Vert X \tilde{v} \Vert_{\psi_2} = C^2$. Hence for each such $\tilde{u}$, $\tilde{v}$ and for all $t > 0$, with probability at least $1- 2 \exp(-t)$, 
\begin{equation}
\left\vert \tilde{u}^T (\hat{\Sigma} - I) \tilde{v} \right\vert \leq 2 C^2 \sqrt{\frac{8t}{n}} + \frac{2 C^2 t}{n}.
\end{equation}
It follows that for all $t > 0$, with probability at least $1- \exp(-t)$
\begin{equation}
\underset{\tilde{u} \in \mathcal{S}_\varepsilon, \tilde{v} \in \mathcal{S}_\varepsilon}{\max} \ \left\vert  \tilde{u}^T (\hat{\Sigma} - I) \tilde{v} \right\vert \leq 2 C^2 \sqrt{\frac{8(t + \log(2 \vert \mathcal{S}_\varepsilon \vert^2 ) )}{n}} + 2 C^2 \left( \frac{t + \log( 2 \vert \mathcal{S}_\varepsilon \vert^2 }{n} \right).
\end{equation}
We now choose
\begin{equation}
\varepsilon := \frac{\sqrt{6} - 2}{2}.
\end{equation}
Then
\begin{equation}
1-2\varepsilon - \varepsilon^2 = \frac{1}{2}.
\end{equation}
Moreover,
\begin{equation}
1 + \frac{2}{\varepsilon} = 1 + \frac{4}{\sqrt{6} - 2} = 1 + \frac{4(\sqrt{6} + 2)}{2} = 2(1 + \sqrt{6}).
\end{equation}
Thus,
\begin{equation}
2 \vert \mathcal{S}_\varepsilon \vert^2 \leq 2 (2(1 + \sqrt{6}))^{2s}.
\end{equation}
But then
\begin{equation}
\log (2 \vert \mathcal{S}_\varepsilon \vert^2) \leq \log 2 + 2s \log(2(1+ \sqrt{6})) \leq \log 2 + 4s.
\end{equation}
\end{proof}

\begin{corollary}
Let $X$ be a row vector in $\mathbb{R}^s$ and $\Sigma_X := \mathbb{E} X^TX$ be arbitrary. We now define 
\begin{equation}
C := \underset{\Vert \Sigma_X^{1/2} u\Vert_2 = 1}{\sup} \ \Vert X u \Vert_{\psi_2}
\end{equation}
and we get for all $t > 0$ that with probability at least $1 - \exp(-t)$
\begin{align}
&\underset{\Vert \Sigma_X^{1/2} u \Vert_2 =1 , \Vert \Sigma_X^{1/2} v \Vert_2 = 1}{\sup} \ \left\vert u^T(\hat{\Sigma} - \Sigma_X) v \right\vert \nonumber \\
&\leq 4 C^2 \sqrt{\frac{8(t+\log 2 + 4s)}{n}} + 4 C^2 \left( \frac{t + \log 2 + 4s}{n} \right).
\end{align}
\end{corollary}

The following lemma about random quadratic forms may be of interest in itself. The following lemma is in its core part a variant of Lemma 15 in \cite{loh2011high}. We use a different technique that involves the Transfer Principle from \cite{oliveira2016lower}.

\begin{lemma}
\label{lemma:quadraticform}
We have for all $u \in \mathbb{R}^p$ and for all $t > 0$
\begin{align}
\left\vert u^T (\hat{\Sigma} - \Sigma_X) u  \right\vert &\leq 12C^2 \sqrt{\frac{8(t + 2 (\log(2p) + 4))}{n}} u^T \Sigma_X u \nonumber \\
&+12 C^2 \sqrt{\frac{16(\log(2p) + 4))}{n}} \Vert u \Vert_1 \sqrt{u^T \Sigma_X u} \nonumber \\
&+ 12 C^2 \left(\frac{t + 2(\log(2p) + 4)}{n} \right) u^T \Sigma_X u \nonumber \\
&+ 12 C^2 \left( \frac{2(\log(2p) + 4)}{n} \right) \Vert u \Vert_1^2
\end{align}
with probability at least $1- \exp(-t)$.
\end{lemma}

\begin{proof}
Let
\begin{equation}
t(s,p) := 4C^2 \sqrt{\frac{8(t+\log 2 + 4s + s \log p)}{n}} + 4 C^2 \left( \frac{t + \log 2 + 4s + s \log p}{n} \right).
\end{equation}
Define the event
\begin{equation}
\mathcal{E} := \left\lbrace \underset{S \subset \left\lbrace 1, \dots, p \right\rbrace: \vert S \vert \leq s}{\sup} \ \left\vert \frac{u_S^T (\hat{\Sigma} - \Sigma_X) v_S}{\Vert X u_S \Vert_2 \Vert	X v_S \Vert_2} \right\vert \leq t(s,p)  \right\rbrace.
\end{equation}
Then on $\mathcal{E}$, for all $u$
\begin{equation}
u_S^T(\hat{\Sigma} - \Sigma_X) u_S \geq - t(s,p) u_S^T \Sigma_X u_S
\end{equation}
or 
\begin{equation}
u_S^T \left( \hat{\Sigma} - \left(1 - t(s,p) \right) \Sigma_X \right) u_S \geq 0.
\end{equation}
By the transfer principle (\cite{oliveira2016lower}), on $\mathcal{E}$, for all $u \in \mathbb{R}^p$
\begin{equation}
u^T \left( \hat{\Sigma} - \left( 1- t(s, p) \right) \right)u \geq - \underset{j}{\max} \ \hat{B}_{j,j} \Vert u \Vert_1^2/(s -1),
\end{equation}
where 
\begin{equation}
\hat{B} = \hat{\Sigma} - (1 - t(s,p)) \Sigma_X.
\end{equation}
We have, using the previous corollary, on $\mathcal{E}$
\begin{equation}
\underset{\vert S \vert, \Vert u_S \Vert_2 = 1}{\sup} \ u_S^T \left( \hat{\Sigma} - ( 1- t(s,p)) \Sigma_X \right) u_S \leq 2 t(s,p).
\end{equation}
We therefore find on $\mathcal{E}$, for all $u \in \mathbb{R}^p$,
\begin{equation}
u^T \left( \hat{\Sigma} - (1-t(s,p)) \Sigma_X \right) u \geq - 2t(s,p) \Vert u \Vert_1^2/(s-1).
\end{equation}
But then on $\mathcal{E}$
\begin{align*}
u^T(\hat{\Sigma} - \Sigma_X) u &\geq - t(s,p) u^T \Sigma_X u - 2 t(s,p) \Vert u \Vert_1^2/(s-1) \\
&= - t(s,p) \left(u^T \Sigma_X u + 2 \Vert u \Vert_1^2/(s-1) \right).
\end{align*}
The same exercise can be done to find that on $\mathcal{E}$, also, for all $u \in \mathbb{R}^p$
\begin{equation}
u^T(\Sigma_X - \hat{\Sigma}) u \geq - t(s,p) \left( u^T \Sigma_X u + 2 \Vert u \Vert_1^2/(s-1) \right).
\end{equation}
Therefore, on $\mathcal{E}$ for all $u \in \mathbb{R}^p$
\begin{equation}
\left\vert u^T (\hat{\Sigma} - \Sigma_X) u \right\vert \leq t(s,p) \left( u^T \Sigma_X u + 2 \Vert u \Vert_1^2/(s-1) \right).
\end{equation}
Therefore, on $\mathcal{E}$ for all $u \in \mathbb{R}^p$ such that $\Vert u \Vert_1^2 \leq (s-1) u^T \Sigma_X u$ we have
\begin{equation}
\left\vert u^T (\hat{\Sigma} - \Sigma_X) u \right\vert \leq 3t (s,p) u^T \Sigma_X u.
\end{equation}
Consider for $k \in \mathbb{N}$ the event
\begin{align*}
\mathcal{F}_k :=& \left\lbrace \underset{ u: \Vert u \Vert_1^2 \leq k}{\sup} \ \left\vert u^T (\hat{\Sigma} - \Sigma_X) u \right\vert \right. \\
&\geq 12C^2 \sqrt{\frac{8(t+\log(2p) + 4 + (\log p + 4) k) }{n}} \\
&\left. + 12 C^2 \left( \frac{t + \log (2p ) + (\log p + 4) k}{n} \right) \right\rbrace.
\end{align*}
We have shown that
\begin{equation}
\mathbb{P} (\mathcal{F}_k) \leq \exp(-t).
\end{equation}
We have for $k \geq 2$
\begin{equation}
2 (\log(2p) + 4)(k - 1) \geq (\log(2p) + 4) k = (\log p + 4) k + k \log 2.
\end{equation}
We also have the partition
\begin{equation}
\left\{ \Vert u \Vert_1^2 \leq 1 \right\} \cup \left\{ 1 < \Vert u \Vert_1^2 \leq 2 \right\} \cup \dots \cup \left\{ k -1 < \Vert u \Vert_1^2 \leq k \right\} \cup \dots .
\end{equation}
If for some $k \geq 2$, it holds that $\Vert u \Vert_1^2 > k - 1$, then the event
\begin{align*}
\left\vert u^T (\hat{\Sigma} - \Sigma_X) u \right\vert &\geq 12 C^2 \sqrt{\frac{8(t + \log(2p) + 4 + 2 \Vert u \Vert_1^2(\log(2p) + 4))}{n}} \\
&+ 12 C^2 \left( \frac{t + \log(2p) + 4 + 2 \Vert u \Vert_1^2 (\log(2p) + 4)}{n} \right)
\end{align*}
implies
\begin{align*}
\left\vert u^T (\hat{\Sigma} - \Sigma_X) u \right\vert &\geq 12 C^2 \sqrt{\frac{8(t+\log(2p) + 4 + k (\log(2p) + 4))}{n}} \\
&+ 12 C^2 \left( \frac{t + \log(2p) + 4 + k(\log(2p) + 4)}{n} \right).
\end{align*}
Hence the event
\begin{align*}
&\exists u : u^T \Sigma_X u = 1: \\
&\left\vert u^T(\hat{\Sigma} - \Sigma_X) u \right\vert &\geq 12 C^2 \sqrt{\frac{8(t+2(\Vert u \Vert_1^2 + 1)(\log(2p) + 4))}{n}} \\
&&+ 12 C^2 \left( \frac{t + 2 (\Vert u \Vert_1^2 + 1)(\log(2p) + 4)}{n} \right)
\end{align*}
has probability at most
\begin{equation}
\sum_{k = 1}^K \exp(-(t+k \log 2)) \leq \exp(-t) \sum_{k = 1}^\infty = \exp(-t).
\end{equation}
In other words, we have shown that the event
\begin{align*}
\forall u : \\
\left\vert u^T (\hat{\Sigma} - \Sigma_X) u \right\vert &\leq 12 C^2 \sqrt{\frac{8(t + 2 (\Vert u \Vert_1^2/(u^T \Sigma_X u) + 2)(\log(2p) + 4))}{n}} u^T \Sigma_X u \\
&+12 C^2 \left( \frac{t + 2(\Vert u \Vert_1^2/(u^T \Sigma_X u ) + 1)(\log(2p) + 4) }{n} \right) u^T \Sigma_X u
\end{align*}
has probability at least $1- \exp(-t)$. It follows that with probability at least $1- \exp(-t)$
\begin{align*}
\forall u: \\
\left\vert u^T(\hat{\Sigma} - \Sigma_X) u \right\vert &\leq 12 C^2 \sqrt{\frac{8(t + 2(\log(2p) + 4))}{n}} u^T \Sigma_X u \\
&+ 12 C^2 \sqrt{\frac{16(\log(2p) + 4))}{n}} \Vert u \Vert_1 \sqrt{u^T \Sigma_X u} \\
&+ 12 C^2 \left( \frac{t + 2(\log(2p) + 4)}{n} \right) u^T \Sigma_X u \\
&+ 12 C^2 \left( \frac{2 (\log(2p) + 4)}{n} \right) \Vert u \Vert_1^2.
\end{align*}
\end{proof}

\subsection{Proofs of the lemmas in Subsection \ref{ss:correctedlinreg}}

\begin{proof}[Proof of Lemma \ref{lemma:twopointcorrreg}]
Let $\beta^t$ be an intermediate point, i.e. $\beta^t = t \beta + (1-t) \beta' \in \mathcal{B}, t \in [0,1]$ (since $\mathcal{B}$ is convex). Then
\begin{align*}
R(\beta) - R(\beta') &= \dot{R}(\beta')^T(\beta - \beta') + (\beta - \beta')^T \ddot{R}(\beta^t) (\beta- \beta') \\
&=  \dot{R}(\beta')^T(\beta - \beta') + (\beta - \beta')^T \Sigma_X (\beta - \beta') \\
&= \dot{R}(\beta')^T (\beta - \beta') + G(\tau(\beta - \beta')),
\end{align*}
where $\tau(\beta - \beta') := \Vert \Sigma^{1/2}_X (\beta - \beta') \Vert_2$ and $G(u) = u^2$.
\end{proof}

\begin{proof}[Proof of Lemma \ref{lemma:quadraticformcorrlinreg}]
The result follows from Lemma \ref{lemma:quadraticform} by noticing that
\begin{equation}
\hat{\Gamma}_{\text{add}} - \Sigma_X = \frac{Z^TZ}{n} - \Sigma_Z.
\end{equation}
\end{proof}

\begin{proof}[Proof of Lemma \ref{lemma:effectivesparscorrlin}]
We have that for $\Vert \beta_S \Vert_0 = s$
\begin{align*}
\Vert \beta_S \Vert_1^2 \leq s \Vert \beta_S \Vert_2^2 \leq s \frac{\Lambda_{\min}(\Sigma_X)\Vert \beta \Vert_2^2}{\Lambda_{\min}(\Sigma_X)} \leq \frac{s}{\Lambda_{\min}(\Sigma_X)} \Vert \Sigma_X^{1/2} \beta \Vert_2^2 = \frac{s}{\Lambda_{\min}(\Sigma_X)} \tau(\beta)^2.
\end{align*}
Hence, 
\begin{align*}
\frac{\Vert \beta_S \Vert_1}{\tau(\beta)} \leq \sqrt{\frac{s}{\Lambda_{\min}(\Sigma_X)}}.
\end{align*}
\end{proof}

\begin{proof}[Proof of Lemma \ref{lemma:covariance}]
We have for all $ u \in \mathbb{R}^p$
\begin{align*}
u^T \Sigma_Z u &= u^T (\Sigma_X + \Sigma_W) u \\
&= u^T \Sigma_X u + u^T \Sigma_W u \\
&\leq G(\tau(u)) + \frac{\Lambda_{\max}(\Sigma_W)}{\Lambda_{\min}(\Sigma_X)} \Lambda_{\min}(\Sigma_X) u^T u \\
&\leq G(\tau(u)) + \frac{\Lambda_{\max}(\Sigma_W)}{\Lambda_{\min}(\Sigma_X)} u^T \Sigma_X u \\
&\leq \left( 1 + \frac{\Lambda_{\max}(\Sigma_W)}{\Lambda_{\min}(\Sigma_X)} \right) G(\tau(u)).
\end{align*}
\end{proof}

\begin{proof}[Proof of Lemma \ref{lemma:correctedlinregsubgauss}]
We have that
\begin{align*}
\left\vert \beta^{*^T} (\hat{\Gamma}_{\text{add}} - \Sigma_X) (\beta' - \beta^*) \right\vert &= \left\vert \beta^{*^T} \left( \frac{Z^TZ}{n} - \Sigma_Z \right) (\beta' - \beta^*) \right\vert \\
&\leq \left\Vert \left( \frac{Z^TZ}{n} - \Sigma_Z \right) \beta^*  \right\Vert_\infty \Vert \beta' - \beta^* \Vert_1.
\end{align*}
We notice that
\begin{align*}
\left\Vert \left( \frac{Z^TZ}{n} - \Sigma_Z \right) \beta^*  \right\Vert_\infty = \underset{1 \leq j \leq p}{\max} \ \left\vert e_j^T \left( \frac{Z^T Z}{n} - \Sigma_Z \right) \beta^* \right\vert.
\end{align*}
For all $j = 1, \dots p$ and for all $t>0$ we have
\begin{align*}
\left\vert e_j^T \left( \frac{Z^T Z}{n} - \Sigma_Z \right) \beta^* \right\vert \leq 8 \Vert \beta^* \Vert_2 C_Z^2 \sqrt{\frac{2t}{n}} + 4 \Vert \beta^*\Vert_2 C_Z^2 \frac{t}{n}.
\end{align*}
with probability at least $1- \exp(-t)$. By the union bound we conclude that for all $t> 0$ the event
\begin{align*}
\left\Vert \left( \frac{Z^TZ}{n} - \Sigma_Z \right) \beta^* \right\Vert_\infty \leq 8 \Vert \beta^* \Vert_2 C_Z^2 \sqrt{\frac{2(t+\log p)}{n}} + 4 \Vert \beta^* \Vert_2 C_Z^2 \frac{t + \log p}{n}
\end{align*}
has probability at least $1 - \exp(-t)$.

We also have that
\begin{align*}
\hat{\gamma}_{\text{add}} - \Sigma_X \beta^0 &= \frac{1}{n} Z^T Y - \Sigma_X \beta^0 \\
&= \left( \frac{X^TX}{n} - \Sigma_X  \right)\beta^0 + \frac{W^T X}{n} \beta^0 + \frac{Z^T \varepsilon}{n}.
\end{align*}
For the first term we have for all $j = 1, \dots, p$ and all $t> 0$ that the event
\begin{align*}
\left\vert e_j^T \left( \frac{X^TX}{n} - \Sigma_X \right) \beta^0 \right\vert \leq 8 C_X^2 \Vert \beta^0 \Vert_2 \sqrt{\frac{2 t}{n}} + 4 C_X^2 \Vert \beta^0 \Vert_2 \frac{t}{n}.
\end{align*}
has probability at least $1- \exp(-t)$. By the union bound we have that the event
\begin{align*}
\left\Vert \left( \frac{X^T X}{n} - \Sigma_X \right) \beta^0 \right\Vert_\infty \leq 8 C_X^2 \Vert \beta^0 \Vert_2 \sqrt{\frac{2(t + \log p)}{n}} + 4 C_X^2 \Vert \beta^0 \Vert_2 \frac{t + \log p}{n}
\end{align*}
has probability at least $1 - \exp(-t)$.

For the second term we have for all $j = 1, \dots, p$ and all $t > 0$ that the event
\begin{align*}
\left\vert e_j^T \left( \frac{W^T X}{n} \right) \beta^0 \right\vert \leq 8 C_W C_X \Vert \beta^0 \Vert_2 \sqrt{\frac{2 t}{n}} + 4 C_W C_X \Vert \beta^0 \Vert_2 \frac{t}{n}
\end{align*}
has probability at least $1- \exp(-t)$. By the union bound we have that the event
\begin{align*}
\left\Vert \frac{W^TX}{n} \beta^0 \right\Vert_\infty \leq 8C_W C_X \Vert \beta^0 \Vert_2 \sqrt{\frac{2(t + \log p)}{n}} + 4 C_W C_X \Vert \beta^0 \Vert_2 \frac{t + \log p}{n}
\end{align*}
has probability at least $1 - \exp(-t)$.

Finally, the we have for the third term for all $j = 1, \dots, p$ and all $t > 0$ that the event
\begin{align*}
\left\vert e_j^T \frac{Z^T \varepsilon}{n} \right\vert \leq 8 C_Z C_\varepsilon \sqrt{\frac{2t}{n}} + 4 C_Z C_\varepsilon \frac{t}{n}
\end{align*}
has probability at least $1- \exp(-t)$. By the union bound we have that the event
\begin{align*}
\left\Vert \frac{Z^T \varepsilon}{n} \right\Vert_\infty \leq 8 C_Z C_\varepsilon \sqrt{\frac{2 (t + \log p)}{n}} + 4 C_Z C_\varepsilon \frac{t + \log p }{n}
\end{align*}
has probability at least $1- \exp(-t)$.

Combining these bound proves that the event
\begin{align*}
&\left\Vert \left( \hat{\Gamma}_{\text{add}} - \Sigma_X \right) \beta^* \right\Vert_\infty + \left\Vert \hat{\gamma}_{\text{add}} - \Sigma_X \beta^0 \right\Vert_\infty \\
&\leq 16 ( C_Z^2 \Vert \beta^* \Vert_2 + C_X^2 \Vert \beta^0 \Vert_2 + C_W C_X \Vert \beta^0 \Vert_2 + C_Z C_\varepsilon) \left(2 \sqrt{\frac{2t + \log p}{n}} + \frac{t + \log p}{n} \right) 
\end{align*}
has probability at least $1- 4 \exp(-t)$.
\end{proof}

\begin{proof}[Proof of Lemma \ref{lemma:additivenoiseempirical}]
The result follows by combining Lemma \ref{lemma:quadraticformcorrlinreg}, Lemma \ref{lemma:covariance}, and Lemma \ref{lemma:correctedlinregsubgauss}.

In fact, by applying Young's inequality with $\zeta > 0$ to the result of Lemma \ref{lemma:quadraticformcorrlinreg} we have 

\begin{align*}
&\left\vert (\beta' - \beta^*)^T (\hat{\Gamma}_{\text{add}} - \Sigma_X) (\beta' - \beta^*) \right\vert \\
&\leq 12 C_Z^2 \left(\frac{8(t + 2 (\log(2p) + 4))}{2n \zeta} + \frac{\zeta}{2} \right) (\beta' - \beta^*)^T \Sigma_X (\beta' - \beta^*) \\
&+ 12 C_Z^2 \left( \frac{16(\log(2p) + 4)}{2n \zeta} \Vert \beta' - \beta^* \Vert_1^2 + \frac{\zeta}{2} (\beta' - \beta^*)^T \Sigma_Z (\beta' - \beta^*) \right) \\
&+12 C_Z^2 \left( \frac{t + 2 (\log(2p) + 4)}{n} \right) (\beta' - \beta^*)^T \Sigma_X (\beta' - \beta^*) \\
&+ 12 C_Z^2 \left( \frac{2 (\log(2p) + 4)}{n} \right) \Vert \beta' - \beta^* \Vert_1^2 \\
&\leq 12 C_Z^2 \left( \frac{4t + 8 \log(2p) + 16}{n} \right) \left(\frac{1}{\zeta} + 1 \right) \Lambda_{\min}(\Sigma_X)^{-1} ( \Lambda_{\min}(\Sigma_X) \\
&+ \Lambda_{\max}(\Sigma_W)) G(\tau(\beta' - \beta^*)) + 12 C_Z^2 \zeta \Lambda_{\min}(\Sigma_X)^{-1} ( \Lambda_{\min}(\Sigma_X) \\ 
&+ \Lambda_{\max}(\Sigma_W)) G(\tau(\beta' - \beta^*)) \\
&+ 12 C_Z^2 \left(\frac{16(\log (2p) + 4)}{2 n \zeta} + \frac{2(\log(2p) + 4)}{n} \right) \Vert \beta' - \beta^* \Vert_1^2.
\end{align*}
With $t = \log(2p)$ we have
\begin{align*}
&\left\vert (\beta' - \beta^*)^T (\hat{\Gamma}_{\text{add}} - \Sigma_X) (\beta' - \beta^*) \right\vert \\
&\leq \gamma G(\tau(\beta' - \beta^*)) \\
&+ 12 C_Z^2 \left(\frac{16(\log (2p) + 4)}{2 n \zeta} + \frac{2(\log(2p) + 4)}{n} \right) \Vert \beta' - \beta^* \Vert_1^2.
\end{align*}
\end{proof}

\subsection{Proofs of the lemmas in Subsection \ref{ss:sparsepca}}

\begin{proof}[Proof of Lemma \ref{lemma:effectivesparsitysparsepca}]
We have that for $\Vert \beta_S \Vert_0 = s$
\begin{align*}
\Vert \beta_S \Vert_1^2 \leq s \Vert \beta_S \Vert_2^2 \leq s \tau(\beta)^2.
\end{align*}
Hence,
\begin{align*}
\frac{\Vert \beta_S \Vert_1}{\tau(\beta)} \leq \sqrt{s}.
\end{align*}
\end{proof}

\begin{proof}[Proof of Lemma \ref{lemma:spcahighprob}]
We have for all $\tilde{\beta} \in \mathcal{B}$
\begin{align*}
&\left( \dot{R}_n(\tilde{\beta}) - \dot{R}(\tilde{\beta}) \right)^T(\tilde{\beta} - \beta^\star) \\
&= \tilde{\beta}^T (\Sigma_X - \hat{\Sigma}) (\tilde{\beta} - \beta^\star) \\
&=(\tilde{\beta} - \beta^\star)^T(\Sigma_X - \hat{\Sigma})(\tilde{\beta} - \beta^\star) + \beta^{\star^T}(\Sigma_X - \hat{\Sigma}) (\tilde{\beta} - \beta^\star).
\end{align*}
To bound the first term we invoke Lemma \ref{lemma:quadraticform}: the event
\begin{align*}
&\left\vert (\tilde{\beta} - \beta^\star)^T (\Sigma_X - \hat{\Sigma}) (\tilde{\beta} - \beta^\star) \right\vert \\
&\leq 12 C_X^2 \sqrt{\frac{8(t + 2(\log(2p) + 4))}{n}} (\tilde{\beta} - \beta^\star)^T \Sigma_X (\tilde{\beta} - \beta^\star) \\
&+ 12 C_X^2 \sqrt{ \frac{16(\log(2p) + 4))}{n} } \Vert \tilde{\beta} - \beta^\star \Vert_1 \sqrt{(\tilde{\beta} - \beta^\star)^T \Sigma_X (\tilde{\beta} - \beta^\star)} \\
&+ 12 C_X^2 \left( \frac{t + 2(\log(2p) + 4)}{n} \right) (\tilde{\beta} - \beta^\star)^T \Sigma_X (\tilde{\beta} - \beta^\star) \\
&+ 12 C_X^2 \left( \frac{2(\log(2p) + 4)}{n} \right) \Vert \tilde{\beta} - \beta^\star \Vert_1^2.
\end{align*}
has probability at least $1- \exp(-t)$.

We now apply Young's inequality with a constant $\zeta > 0$:
\begin{align*}
&\left\vert (\tilde{\beta} - \beta^*)^T (\Sigma_X - \hat{\Sigma}) (\tilde{\beta} - \beta^*) \right\vert \\
&\leq 12 C_X^2 \left( \frac{8(t + 2(\log(2p) + 4))}{2n \zeta} + \frac{\zeta}{2} \right) (\tilde{\beta} - \beta^*)^T \Sigma_X (\tilde{\beta} - \beta^*) \\
&+ 12 C_X^2 \left( \frac{16(\log(2p) + 4)}{2n\zeta} \Vert \tilde{\beta} - \beta^* \Vert_1^2 + \frac{\zeta}{2} (\tilde{\beta} -\beta^*)^T\Sigma_X (\tilde{\beta} - \beta^*) \right) \\
&+ 12 C_X^2 \left( \frac{t + 2 (\log(2p) + 4)}{n} \right) (\tilde{\beta} - \beta^*)^T \Sigma_X (\tilde{\beta} - \beta^*) \\
&+12 C_X^2 \left( \frac{2 (\log(2p) + 4)}{n} \right) \Vert \tilde{\beta} - \beta^* \Vert_1^2.
\end{align*}
With $t = \log(2p)$ and using that 
\begin{align*}
&(\tilde{\beta} - \beta^*)^T \Sigma_X (\tilde{\beta} - \beta^*) \\
&\leq \Lambda_{\max}(\Sigma_X) \Vert \tilde{\beta} - \beta^* \Vert_2^2 \\
&= \frac{\Lambda_{\max}(\Sigma_X)}{2\phi_{\max}(\rho - 3 \eta)} 2\phi_{\max}(\rho - 3 \eta) \Vert \tilde{\beta} - \beta^* \Vert_2^2 = \frac{\Lambda_{\max}(\Sigma_X)}{2\phi_{\max}(\rho - 3\eta)} G(\tau(\tilde{\beta} - \beta^*))
\end{align*}
 we obtain
\begin{align*}
&\left\vert (\tilde{\beta} - \beta^*)^T (\Sigma_X - \hat{\Sigma})(\tilde{\beta} - \beta^*) \right\vert \\
&\leq \gamma G(\tau(\tilde{\beta} - \beta^*)) + 12 C_X^2 \left( \frac{16(\log(2p) + 4)}{2n\zeta} + \frac{2 (\log(2p) + 4)}{n} \right) \Vert \tilde{\beta} - \beta^* \Vert_1^2.
\end{align*}

For the second term we notice that
\begin{align*}
\left\Vert (\hat{\Sigma} - \Sigma_X) \beta^\star \right\Vert_\infty = \underset{1 \leq j \leq p}{\max} \ \left\vert e_j^T \left( \hat{\Sigma} - \Sigma_X \right) \beta^\star \right\vert.
\end{align*}
For all $j = 1, \dots, p$ and all $t > 0$ the event
\begin{align*}
\left\vert e_j^T \left( \hat{\Sigma} - \Sigma_X \right) \beta^\star \right\vert \leq 8 C_X^2 (\Vert \beta^0 \Vert_2 + \eta) \sqrt{\frac{2t}{n}} + 4 C_X^2 (\Vert \beta^0 \Vert_2 + \eta) \frac{t}{n}
\end{align*}
has probability at least $1- \exp(-t)$. By the union bound we have that the event
\begin{align*}
\left\Vert  (\hat{\Sigma} - \Sigma_X) \beta^\star \right\Vert_\infty \leq 8 C_X^2 (\Vert \beta^0 \Vert_2 + \eta) \sqrt{\frac{2(t + \log p) }{n}} + 4 C_X^2 (\Vert \beta^0 \Vert_2 + \eta) \frac{t + \log p}{n}
\end{align*}
has probability at least $1- \exp(-t)$.

The result follows by combining Lemma \ref{lemma:twopointpca} with the upper bounds.
\end{proof}

\subsection{Proofs of the lemmas in Subsection \ref{ss:robustreg2}}

\begin{proof}[Proof of Lemma \ref{lemma:twopointrobust}]

Consider the Taylor expansion of the risk
\begin{align*}
R(\beta_1) = R(\beta_2) + \dot{R}(\beta_2)^T(\beta_1 - \beta_2) + (\beta_1 - \beta_2)^T \ddot{R}(\beta^t) (\beta_1 - \beta_2), 
\end{align*}
where $\beta^t = t \beta_1 + (1- t)  \beta_2 \in \mathcal{B}$ and $t \in \left[ 0, 1\right]$.

Hence,

\begin{align*}
R(\beta_1) - R(\beta_2) \geq \dot{R}(\beta_2)^T (\beta_1 - \beta_2) + (\beta_1 - \beta_2)^T \ddot{R}(\beta^t) (\beta_1 - \beta_2).
\end{align*}
For ease of notation we define the event $A_i$ as
\begin{equation*}
A_i = \left\lbrace \vert \varepsilon_i \vert \leq \frac{T}{2} \right\rbrace \cap \left\lbrace \vert X_i (\beta_1 - \beta_2) \vert \leq \frac{T}{8 \eta} \Vert \beta_1 - \beta_2 \Vert_2 \right\rbrace \cap \left\lbrace \vert X_i (\beta_2 - \beta^0) \vert \leq \frac{T}{4} \right\rbrace.
\end{equation*}

We also define the truncation functions $\varphi_t$ and $\psi_t$ as

\begin{equation*}
\varphi_t(u) = \left\lbrace \begin{array}{ll}
u^2, &\text{if} \ \vert u \vert \leq \frac{t}{2}, \\
(t - u)^2, &\text{if} \ \frac{t}{2} \leq \vert u \vert \leq t, \\
0, &\text{if} \ \vert u \vert \geq t,
\end{array} \right.
\ \text{and} \ 
\psi_t(u) = \left\lbrace \begin{array}{ll}
1, &\text{if} \ \vert u \vert \leq \frac{t}{2}, \\
2 - \frac{2}{t} \vert u \vert, &\text{if} \ \frac{t}{2} \leq \vert u \vert \leq t, \\
0, &\text{if} \ \vert u \vert \geq t.
\end{array} \right.
\end{equation*}

and the functions

\begin{align*}
&f(\beta_1, \beta_2) = \frac{1}{n} \sum_{i = 1}^n \varphi_{T \Vert \beta_1 - \beta_2 \Vert_2/8\eta} \left( X_i (\beta_1 - \beta_2) \right) \cdot \psi_{T/2} (\varepsilon_i) \cdot \psi_{T/4} (X_i(\beta_2 - \beta^0)), \\
&\tilde{f}(\beta_1, \beta_2) = \frac{1}{n} \sum_{i = 1}^n (X_i (\beta_1 - \beta_2))^2.
\end{align*}

The second order term of the Taylor expansion of the risk is given by
\begin{align*}
&(\beta_1 - \beta_2)^T \ddot{R}(\beta^t) (\beta_1 - \beta_2) \\
&= \mathbb{E} \left[ \frac{1}{n} \sum_{i=1}^n \ddot{\rho}(Y_i - X_i \beta^t) (X_i(\beta_1 - \beta_2))^2 \right] \\
&= \mathbb{E} \left[ \frac{1}{n} \sum_{i = 1}^n \ddot{\rho}(Y_i - X_i \beta^t) (X_i(\beta_1 - \beta_2))^2 1_{A_i} \right] \\
&+ \mathbb{E} \left[ \frac{1}{n} \sum_{i = 1}^n \ddot{\rho}(Y_i - X_i \beta^t) (X_i (\beta_1 - \beta_2))^2 1_{A_i^c} \right] \\
&\geq \frac{7}{2}\alpha_T \mathbb{E} \left[ \frac{1}{n} \sum_{i =1}^n (X_i (\beta_1 - \beta_2))^2 1_{A_i} \right] - \kappa_2 \mathbb{E} \left[ \frac{1}{n} \sum_{i = 1}^n (X_i (\beta_1 - \beta_2))^2 1_{A_i^c} \right] \\
&+ \kappa_2 \mathbb{E} \left[ \frac{1}{n} \sum_{i = 1}^n (X_i (\beta_1 - \beta_2))^2 1_{A_i} \right] \\
&- \kappa_2 \mathbb{E} \left[ \frac{1}{n} \sum_{i = 1}^n (X_i (\beta_1 - \beta_2))^2 1_{A_i} \right] \\
&= (\frac{7}{2}\alpha_T + \kappa_2) \mathbb{E} \left[ \frac{1}{n} \sum_{i = 1}^n (X_i(\beta_1 - \beta_2))^2 1_{A_i} \right] - \kappa_2 \mathbb{E} \left[ \frac{1}{n} \sum_{i = 1}^n (X_i (\beta_1 - \beta_2))^2 \right] \\
&\geq ( \frac{7}{2} \alpha_T + \kappa_2) \mathbb{E} \left[ f(\beta_1, \beta_2) \right] - \kappa_2 \mathbb{E} \left[ \tilde{f} (\beta_1, \beta_2) \right] \\
&= (\frac{7}{2}\alpha_T + \kappa_2) \mathbb{E} \left[ f(\beta_1, \beta_2) \right] - (\frac{7}{2} \alpha_T + \kappa_2) \mathbb{E} \left[ \tilde{f}(\beta_1, \beta_2) \right] \\
&+ (\frac{7}{2} \alpha_T + \kappa_2) \mathbb{E} \left[ \tilde{f}(\beta_1, \beta_2)  \right] - \kappa_2 \mathbb{E} \left[ \tilde{f}(\beta_1, \beta_2) \right]
\end{align*}
As a consequence of Lemma 7 in \cite{loh2015statistical} we have that
\begin{align*}
(\frac{7}{2}\alpha_T + \kappa_2) \left( \mathbb{E} \left[ \tilde{f} (\beta_1, \beta_2) \right] - \mathbb{E} \left[f(\beta_1, \beta_2) \right] \right) \leq \frac{\alpha_T}{2} \mathbb{E} \left[ \tilde{f}(\beta_1, \beta_2) \right].
\end{align*}
Hence,
\begin{align*}
&(\frac{7}{2}\alpha_T + \kappa_2) \mathbb{E} \left[ f(\beta_1, \beta_2) \right] - (\frac{7}{2}\alpha_T + \kappa_2) \mathbb{E} \left[ \tilde{f}(\beta_1, \beta_2) \right] \\
&+ (\frac{7}{2}\alpha_T + \kappa_2) \mathbb{E} \left[ \tilde{f}(\beta_1, \beta_2)  \right] - \kappa_2 \mathbb{E} \left[ \tilde{f}(\beta_1, \beta_2) \right] \\
&\geq - \frac{\alpha_T}{2} \mathbb{E} \left[ \tilde{f}(\beta_1, \beta_2) \right] + \frac{7}{2}\alpha_T \mathbb{E} \left[ \tilde{f} (\beta_1, \beta_2) \right] \\
&= 3 \alpha_T \mathbb{E} \left[ \tilde{f}(\beta_1, \beta_2) \right] \\
&= 3 \alpha_T (\beta_1 - \beta_2)^T \Sigma_X (\beta_1 - \beta_2) = 3 \alpha_T \Vert \Sigma_X^{1/2} (\beta_1 - \beta_2) \Vert_2^2 = G(\tau(\beta_1 - \beta_2)).
\end{align*}
\end{proof}

\begin{proof}[Proof of Lemma \ref{lemma:robustregremp}]

\begin{align*}
&\left( \dot{R}_n(\beta') - \dot{R}(\beta') \right)^T(\beta^\star - \beta') \\
&= \left(\dot{R}_n(\beta') - \dot{R}_n(\beta^\star) + \dot{R}_n(\beta^\star) - \dot{R}(\beta^\star) + R(\beta^\star) - R(\beta') \right)^T(\beta^\star - \beta') \\
&= \left( \dot{R}_n(\beta') - \dot{R}_n(\beta^\star) \right)^T (\beta^\star - \beta') =: (I) \\
&+ \left( \dot{R}_n(\beta^\star) - \dot{R}(\beta^\star) \right)^T (\beta^\star - \beta') =: (II) \\
&+ \left( \dot{R}(\beta^\star) - \dot{R}(\beta') \right)^T(\beta^\star - \beta') =: (III)
\end{align*}
To bound the first term we make use of the local restricted strong convexity from \cite{loh2015statistical} (Proposition \ref{prop:rsc}).

There, we find that the following holds: there exist $\alpha > 0$ and $\xi \geq 0$ such that
 for all $\beta_1, \beta_2 \in  \mathcal{B}$
\begin{equation*}
\left( \dot{R}_n(\beta_1) - \dot{R}_n(\beta_2) \right)^T (\beta_1 - \beta_2) \geq \alpha \Vert \beta_1 - \beta_2 \Vert_2^2 - \xi \frac{\log p}{n} \Vert \beta_1 - \beta_2 \Vert_1^2.
\end{equation*}
Hence, multiplying the inequality by $(-1)$ we also have for all $\beta_1, \beta_2 \in \mathcal{B}$
\begin{equation*}
\left( \dot{R}_n(\beta_1) - \dot{R}_n(\beta_2) \right)^T (\beta_2 - \beta_1) \leq \xi \frac{\log p}{n} \Vert \beta_1 - \beta_2 \Vert_1^2 - \alpha \Vert \beta_1 - \beta_2 \Vert_2^2.
\end{equation*}
As far as the second term $(II)$ is concerned, we have that the following holds
\begin{align*}
&\left(\dot{R}_n(\beta^\star) - \dot{R}(\beta^\star)  \right)^T (\beta^\star - \beta') \\
&\leq \left\vert \left( \dot{R}_n(\beta^\star) - \dot{R}(\beta^\star) \right)^T (\beta^\star - \beta') \right\vert \\
&\leq \left\Vert \dot{R}_n(\beta^\star) - \dot{R}(\beta^\star) \right\Vert_\infty \Vert \beta^\star - \beta' \Vert_1.
\end{align*}
We have that
\begin{align*}
&\left\Vert \dot{R}_n(\beta^\star) - \dot{R}(\beta^\star) \right\Vert_\infty \\
&= \underset{ 1 \leq j \leq p}{\max} \left\vert \frac{1}{n} \sum_{i = 1}^n \dot{\rho}(Y_i - X_i \beta^\star) X_{ij} - \mathbb{E} \left[ \dot{\rho}(Y_i - X_i \beta^\star) X_{ij} \right] \right\vert 
\end{align*}

By assumption,  $X_{ij}$ is sub-Gaussian. Hence, we have for all $j \in \left\lbrace 1, \dots, p \right\rbrace$ that for all $t > 0$ invoking Lemma 14.16 in \cite{buhlmann2011statistics}

\begin{align*}
\mathbb{P} \left( \left\Vert \dot{R}_n(\beta^\star) - \dot{R}(\beta^\star) \right\Vert_\infty \geq 4 \kappa_1 C_X \sqrt{ \left( t^2 +  \frac{\log (2p)}{n} \right)}  \right) \leq \exp(-nt^2).
\end{align*}

As far as the third term $(III)$ is concerned, we have that

\begin{align*}
&\left( \dot{R}(\beta_1) - \dot{R}(\beta_2) \right)^T(\beta_1 - \beta_2) \\
&\leq \left\vert \left( \dot{R}(\beta_1) - \dot{R}(\beta_2) \right)^T(\beta_1 - \beta_2)  \right\vert \\
&= \left\vert \mathbb{E} \left[ \frac{1}{n} \sum_{i = 1}^n \left( - \dot{\rho}(Y_i - X_i \beta_1) + \dot{\rho}(Y_i - X_i \beta_2) \right) X_i (\beta_1 - \beta_2)\right] \right\vert \\
&\leq \mathbb{E} \left[ \frac{1}{n} \sum_{i = 1}^n \left\vert \dot{\rho}(Y_i - X_i \beta_2) - \dot{\rho}(Y_i - X_i \beta_2) \right\vert \left\vert X_i (\beta_1 - \beta_2)\right\vert \right] \\
&\leq \kappa_2 \frac{1}{n} \sum_{i = 1}^n \mathbb{E} \left[ \left\vert X_i (\beta_1 - \beta_2) \right\vert^2 \right] \\
&= \frac{\kappa_2}{3 \alpha_T} G(\tau(\beta_1 - \beta_2))
\end{align*}
\end{proof}

\subsection{Proofs of the lemmas in Subsection \ref{ss:binaryclass}}

\begin{proof}[Proof of Lemma \ref{lemma:twopointmarginbinlinclass}]
We consider the two term Taylor expansion of $R$ for an intermediate point $\beta^\dagger \in \mathcal{B}$
\begin{align*}
R(\beta_1) = R(\beta_2) + \dot{R}(\beta_2)^T(\beta_1 - \beta_2) + (\beta_1 - \beta_2)^T \ddot{R}(\beta^\dagger)(\beta_1 - \beta_2).
\end{align*}
By adding an subtracting $\ddot{R}(\beta^0)$ we have
\begin{align*}
&R(\beta_1) - R(\beta_2) - \dot{R}(\beta_2)^T(\beta_1 - \beta_2) \\
&= (\beta_1 - \beta_2)^T (\ddot{R}(\beta^\dagger) - \ddot{R}(\beta^0)) (\beta_1 - \beta_2) + (\beta_1 - \beta_2)^T \ddot{R}(\beta^0) (\beta_1 - \beta_2).
\end{align*}
As far as the first term is concerned we have that
\begin{align*}
\ddot{R}(\beta^\dagger) - \ddot{R}(\beta^0) = \mathbb{E} \left[ \frac{1}{n} \sum_{i = 1}^n (g(X_i \beta^\dagger) - g(X_i \beta^0)) X_i^T X_i \right],
\end{align*}
where
\begin{align*}
g(X_i \beta^\dagger) = 2(\sigma'(X_i \beta^\dagger)^2 + (\sigma(X_i \beta^\dagger) - Y_i) \sigma''(X_i \beta^\dagger))
\end{align*}
We see that $g$ is Lipschitz continuous by considering its first derivative
\begin{align*}
\left\vert \frac{d}{du} g(u) \right\vert &= 2 \left\vert 2 \sigma'(u) \sigma''(u) + \sigma'(u) \sigma''(u) + \sigma(u) \sigma'''(u) - Y_i \sigma'''(u) \right\vert \\
&\leq 10,
\end{align*}
where we have used that $\vert \sigma(\cdot) \vert \leq 1, \vert \sigma'(\cdot) \vert \leq 1, \vert \sigma''(\cdot) \vert \leq 1, \vert \sigma''(\cdot) \vert \leq 1, \vert \sigma'''(\cdot) \vert \leq 1$, and $\vert Y_i \vert \leq 1$.
Hence, we have for all $v \in \mathbb{R}^p$ with $\Vert v \Vert_2 = 1$ that
\begin{align*}
\left\vert v^T (\ddot{R}(\beta^\dagger) - \ddot{R}(\beta^0) )v \right\vert & \leq 10 \mathbb{E} \left[ \frac{1}{n} \sum_{i = 1}^n X_i (\beta^\dagger - \beta^0) (v^T X_i^T)^2  \right] \\
&\leq 10 \mathbb{E} \left[ (X_i (\beta^\dagger - \beta^0))^2 \right]^{1/2} \mathbb{E} \left[(v^T X_i^T)^4 \right]^{1/2},  \\
&\leq 10 C_X^3 \underbrace{\Vert \beta^\dagger - \beta^0 \Vert_2}_{\leq \eta}.
\end{align*}
We also have that for all $v \in \mathbb{R}^p$ with $\Vert v \Vert_2 = 1$
\begin{align*}
v^T \ddot{R}(\beta^0) v \geq 2 V \Lambda_{\min} (\Sigma_X).
\end{align*}
Therefore, we conclude that
\begin{align*}
&R(\beta_1) - R(\beta_2) - \dot{R}(\beta_2)^T (\beta_1 - \beta_2) \\
&= (\beta_1 - \beta_2)^T (\ddot{R}(\beta^\dagger) - \ddot{R}(\beta^0)) (\beta_1 - \beta_2) + (\beta_1 - \beta_2)^T \ddot{R}(\beta^0) (\beta_1 - \beta_2) \\
&\geq 2 (V \Lambda_{\min}(\Sigma_X) - 5 C_X^3 \eta) \Vert \beta_1 - \beta_2 \Vert_2^2 \\
&= G(\Vert \beta_1 - \beta_2 \Vert_2).
\end{align*}
\end{proof}

\begin{proof}[Proof of Lemma \ref{lemma:empirprocbinlinclass}]
We first consider the empirical process for $M > 0$:
\begin{align*}
Z_M &= \underset{\tilde{\beta} \in \mathcal{B}: \Vert \tilde{\beta} - \beta^\star \Vert_1 \leq M}{\sup} \ \left\vert \left(  \dot{R}_n(\tilde{\beta}) - \dot{R}(\tilde{\beta})\right)^T(\beta^\star - \tilde{\beta}) \right\vert \\
&= \underset{\tilde{\beta} \in \mathcal{B}: \Vert \tilde{\beta} - \beta^\star \Vert_1 \leq M}{\sup} \ \left\vert - \frac{2}{n} \sum_{i = 1}^n \left[ (Y_i - \sigma(X_i \tilde{\beta})) \sigma'(X_i \tilde{\beta}) X_i \right. \right. \\
&\left. \left. + \mathbb{E} \left[ (Y_i - \sigma(X_i \tilde{\beta})) \sigma'(X_i \tilde{\beta}) X_i \right] \right] (\beta^\star - \tilde{\beta}) \right\vert
\end{align*}
By the symmetrization theorem for i.i.d. Rademacher random variables $\tilde{\varepsilon}_1 , \dots, \tilde{\varepsilon}_n$ we have
\begin{align*}
\mathbb{E} Z_M \leq \mathbb{E} \underset{\tilde{\beta} \in \mathcal{B} : \Vert \tilde{\beta} - \beta^\star \Vert_1 \leq M}{\sup} \ \left\vert \frac{4}{n} \sum_{i = 1}^n \tilde{\varepsilon}_i (Y_i - \sigma(X_i \tilde{\beta})) \sigma'(X_i \tilde{\beta}) X_i (\beta^\star - \tilde{\beta}) \right\vert.
\end{align*}
We define the function
\begin{align*}
f(X_i \tilde{\beta}) = (Y_i - \sigma(X_i \tilde{\beta})) \sigma'(X_i \tilde{\beta}) X_i (\beta^\star - \tilde{\beta}).
\end{align*}
The function $f$ is Lipschitz continuous as one can see by considering its first derivative:
\begin{align*}
\left\vert \frac{d}{du} f(u) \right\vert &= \left\vert \sigma''(u) Y_i X_i \beta^\star - \sigma'(u)^2 X_i \beta^\star - \sigma(u) \sigma''(u) X_i \beta^\star - Y_i \sigma''(u) u \right. \\ 
&\left.\phantom{aaa}- Y_i \sigma'(u)   + \sigma'(u)^2 u + \sigma(u) \sigma''(u) u + \sigma(u) \sigma'(u) \right\vert \\
&\leq 6 K_2 + 2.
\end{align*}
By the contraction Theorem (\cite{ledoux1991probability}) using that $f(X_i \beta^\star) = 0$ and by the dual norm inequality we then have 
\begin{align*}
\mathbb{E} Z_M &\leq \mathbb{E} \underset{\tilde{\beta} \in \mathcal{B} : \Vert \tilde{\beta} - \beta^\star \Vert_1 \leq M}{\sup} \ \left\vert \frac{4}{n} \sum_{i = 1}^n \tilde{\varepsilon}_i  (f(X_i \tilde{\beta}) - f(X_i \beta^\star)) \right\vert \\
&\leq 8(6 K_2 + 2) \mathbb{E} \underset{\tilde{\beta} \in \mathcal{B}: \Vert \tilde{\beta} - \beta^\star \Vert_1 \leq M}{\sup} \ \left\vert \frac{1}{n} \sum_{i = 1}^n \tilde{\varepsilon}_i X_i (\beta^\star - \tilde{\beta}) \right\vert \\
&\leq 8(6 K_2 + 2) M \mathbb{E} \left\Vert \frac{1}{n} \sum_{i = 1}^n \tilde{\varepsilon}_i X_i \right\Vert_\infty.
\end{align*}

We notice that by Lemma \ref{lemma:productofsubgauss} the following holds:
\begin{align*}
\mathbb{E} \left[ \exp \left( \frac{\vert \tilde{\varepsilon}_i X_{ij} \vert}{\Vert \tilde{\varepsilon}_i \Vert_{\psi_2} \Vert X_{ij} \Vert_{\psi_2}} \right) \right] - 1 \leq 1.
\end{align*}
Therefore, expanding the exponential we have
\begin{align*}
&\mathbb{E} \left[ \exp \left( \frac{\vert \tilde{\varepsilon}_i X_{ij} \vert}{\Vert \tilde{\varepsilon}_i \Vert_{\psi_2} \Vert X_{ij} \Vert_{\psi_2} } \right) \right] - 1\\
&= \sum_{k = 1}^\infty \frac{1}{k!} \frac{\vert \tilde{\varepsilon}_i X_{ij} \vert^k}{\Vert \tilde{\varepsilon}_i \Vert_{\psi_2}^k \Vert X_{ij} \Vert_{\psi_2}^k} \\
&= \frac{1}{m!} \frac{\vert \tilde{\varepsilon}_i X_{ij} \vert^m}{\Vert \tilde{\varepsilon}_i \Vert_{\psi_2}^m \Vert X_{ij} \Vert_{\psi_2}^m } + \sum_{k = 1, k \neq m}^\infty \frac{1}{k!} \frac{\vert \tilde{\varepsilon}_i X_{ij} \vert^k}{\Vert \tilde{\varepsilon}_i \Vert_{\psi_2}^k \Vert X_{ij} \Vert_{\psi_2}^k}
&\leq 1.
\end{align*}
This implies for all $m \geq 1$ that
\begin{align*}
\frac{1}{m!} \frac{\vert \tilde{\varepsilon}_i X_{ij} \vert^m }{\Vert \tilde{\varepsilon}_i \Vert_{\psi_2}^m \Vert X_{ij} \Vert_{\psi_2}^m} \leq 1.
\end{align*}
Hence,
\begin{align*}
\vert \tilde{\varepsilon}_i X_{ij} \vert^m \leq m! \Vert \tilde{\varepsilon}_i \Vert_{\psi_2}^m \Vert X_{ij} \Vert_{\psi_2}^m = \frac{m!}{2} \Vert \tilde{\varepsilon}_i \Vert_{\psi_2}^{m-2} \Vert X_{ij} \Vert_{\psi_2}^{m- 2} (2 \Vert \tilde{\varepsilon}_i \Vert_{\psi_2}^2 \Vert X_{ij} \Vert_{\psi_2}^2).
\end{align*}
Dividing the previous equation by $(\sqrt{2} \Vert \tilde{\varepsilon}_i \Vert_{\psi_2} \Vert X_{ij} \Vert_{\psi_2})^m$ we obtain
\begin{align*}
\left( \frac{\vert \tilde{\varepsilon}_i X_{ij} \vert }{\sqrt{2} \Vert \tilde{\varepsilon}_i \Vert_{\psi_2} \Vert X_{ij} \Vert_{\psi_2}} \right)^m \leq \frac{m!}{2} \left( \frac{\Vert \tilde{\varepsilon}_i \Vert_{\psi_2} \Vert X_{ij} \Vert_{\psi_2}}{\sqrt{2} \Vert \tilde{\varepsilon}_i \Vert_{\psi_2} \Vert X_{ij} \Vert_{\psi_2} } \right)^{m-2} = \frac{m!}{2} \left( \frac{1}{\sqrt{2}} \right)^{m-2}.
\end{align*}
We also notice that
\begin{align*}
\mathbb{E} \left[ \tilde{\varepsilon}_i X_{ij} \right] = \mathbb{E} \tilde{\varepsilon}_i \mathbb{E} X_{ij} = 0.
\end{align*}
By Lemma 14.12 in \cite{buhlmann2011statistics} we then have that
\begin{align*}
\mathbb{E} \left\Vert \frac{1}{n} \sum_{i = 1}^n \frac{ \tilde{\varepsilon}_i X_{ij}  }{\sqrt{2} \Vert \tilde{\varepsilon}_i \Vert_{\psi_2} \Vert X_{ij} \Vert_{\psi_2}}   \right\Vert_\infty \leq \frac{\log(p + 1)}{\sqrt{2} n} + \sqrt{\frac{2 \log(p + 1)}{n}}
\end{align*}
or equivalently
\begin{align*}
\mathbb{E} \left\Vert \frac{1}{n} \sum_{i = 1}^n  \tilde{\varepsilon}_i X_{ij}  \right\Vert_\infty &\leq \Vert \tilde{\varepsilon}_i \Vert_{\psi_2} \Vert X_{ij} \Vert_{\psi_2} \frac{\log(p + 1)}{n} + \sqrt{\frac{4 \Vert \tilde{\varepsilon}_i \Vert_{\psi_2}^2 \Vert X_{ij} \Vert_{\psi_2}^2 \log(p + 1)}{n}} \\
&=: \tilde{\lambda}_1.
\end{align*}
In view of applying Bousquet's concentration inequality (\cite{bousquet2002bennett}) in the form given in \cite{geer2015} in Corollary 16.1 we notice that
\begin{align*}
&\underset{\tilde{\beta} \in \mathcal{B}: \Vert \tilde{\beta} - \beta^\star \Vert_1 \leq M}{\sup} \ \text{Var}(f(X_i\tilde{\beta})) \\
&\leq \underset{\tilde{\beta} \in \mathcal{B}: \Vert \tilde{\beta} - \beta^\star \Vert_1 \leq M}{\sup} \ \mathbb{E} \left[ f(X_i \tilde{\beta})^2 \right] \\
&= \underset{\tilde{\beta} \in \mathcal{B}: \Vert \tilde{\beta} - \beta^\star \Vert_1 \leq M}{\sup} \ \mathbb{E} \left[ (Y_i - \sigma(X_i \tilde{\beta}))^2 \sigma'(X_i \tilde{\beta})^2 (X_i (\beta^\star - \tilde{\beta}))^2 \right] \\
&\leq \underset{\tilde{\beta} \in \mathcal{B}: \Vert \tilde{\beta} - \beta^\star \Vert_1 \leq M}{\sup} \ \mathbb{E} \left[ \frac{(X_i (\beta^\star - \tilde{\beta}))^2}{\Vert \beta^\star - \tilde{\beta} \Vert_2^2} \right] \Vert \beta^\star - \tilde{\beta} \Vert_2^2 \\
&\leq \underset{\tilde{\beta} \in \mathcal{B}: \Vert \tilde{\beta} - \beta^\star \Vert_1 \leq M}{\sup} \ C_X^2 \Vert \beta^\star - \tilde{\beta} \Vert_1^2 \\
&\leq C_X^2 M^2.
\end{align*}
We also have that
\begin{align*}
\Vert f \Vert_\infty \leq 2 K_2.
\end{align*}
Hence, by Bousquet's concentration inequality we obtain for all $ t > 0$
\begin{align*}
\mathbb{P} \left( Z_M \geq 2 \mathbb{E} Z_M + C_X M \sqrt{\frac{2 t}{n}} + \frac{8 K_2 t}{n} \right) \leq \exp(-t).
\end{align*}
Therefore,
\begin{align*}
\mathbb{P} \left( Z_M \geq 16(6K_2 + 2) M \tilde{\lambda}_1 + C_X M \sqrt{\frac{2t}{n}} + \frac{8 K_2 t}{n} \right) \leq \exp(-t).
\end{align*}
To simplify the notation, we define
\begin{align*}
\tilde{\lambda}_2 (t) = 16(6K_2 + 2) \tilde{\lambda}_1 + C_X \sqrt{\frac{2t}{n}}.
\end{align*}
Hence,
\begin{align*}
\mathbb{P} \left(Z_M \geq M \tilde{\lambda}_2(t) + \frac{8 K_2 t}{n} \right) \leq \exp(-t).
\end{align*}
Define
\begin{align*}
Z(\tilde{\beta}, \beta^\star) := \left\vert \left( \dot{R}_n(\tilde{\beta}) - \dot{R}(\tilde{\beta}) \right)^T(\beta^\star - \tilde{\beta}) \right\vert
\end{align*}
Hence,
\begin{align*}
&\mathbb{P} \left( \exists \tilde{\beta} \in \mathcal{B} : \Vert \tilde{\beta} - \beta^\star \Vert_1 \leq \frac{1}{n} : Z(\tilde{\beta}, \beta^\star) \geq 2 \Vert \tilde{\beta} - \beta^\star \Vert_1 \tilde{\lambda}_2(t) + \frac{\tilde{\lambda}_2(t)}{n} + \frac{8 K_2 t}{n} \right) \\
&\leq \mathbb{P} \left( \exists \tilde{\beta} \in \mathcal{B} : \Vert \tilde{\beta} - \beta^\star \Vert_1 \leq \frac{1}{n} : Z(\tilde{\beta}, \beta^\star) \geq  \frac{\tilde{\lambda}_2(t)}{n} + \frac{8 K_2 t}{n} \right) \\
&\leq \exp(-t).
\end{align*}
We also have
\begin{align*}
&\mathbb{P} \left( \exists \tilde{\beta} \in \mathcal{B} : \frac{1}{n} < \Vert \tilde{\beta} - \beta^\star \Vert_1 \leq 2 \sqrt{p} \eta : \right. \\
&\left.\phantom{aaaaaaaaaa} Z(\tilde{\beta}, \beta^\star) \geq 2 \Vert \tilde{\beta} - \beta^\star \Vert_1 \tilde{\lambda}_2(t) + \frac{\tilde{\lambda}_2(t)}{n} + \frac{8 K_2 t}{n} \right) \\
&\leq \sum_{j = 0}^{\lceil \log_2(\eta n \sqrt{p}) \rceil} \mathbb{P} \left( \exists \tilde{\beta} \in \mathcal{B} : \frac{2^j}{n} < \Vert \tilde{\beta} - \beta^\star \Vert_1 \leq \frac{2^{j+1}}{n} : \right. \\
&\left. \phantom{aaaaaaaaaaaaaa} Z(\tilde{\beta}, \beta^\star) \geq 2 \Vert \tilde{\beta} - \beta^\star \Vert_1 \tilde{\lambda}_2(t) + \frac{\tilde{\lambda}_2(t)}{n} + \frac{8 K_2 t}{n} \right) \\
&\leq \sum_{j = 0}^{\lceil \log_2(\eta n \sqrt{p}) \rceil} \mathbb{P} \left( \exists \tilde{\beta} \in \mathcal{B} : \frac{2^j}{n} < \Vert \tilde{\beta} - \beta^\star \Vert_1 \leq \frac{2^{j+1}}{n} : \right. \\
&\left. \phantom{aaaaaaaaaaaaaa} Z(\tilde{\beta}, \beta^\star) \geq \frac{2^{j+1}}{n}\tilde{\lambda}_2(t) + \frac{\tilde{\lambda}_2(t)}{n} + \frac{8 K_2 t}{n} \right) \\
&\leq \sum_{j = 0}^{\lceil \log_2(\eta n \sqrt{p}) \rceil} \mathbb{P} \left( \exists \tilde{\beta} \in \mathcal{B} :  \Vert \tilde{\beta} - \beta^\star \Vert_1 \leq \frac{2^{j+1}}{n} : \right. \\
&\left. \phantom{aaaaaaaaaaaaaaaa} Z(\tilde{\beta}, \beta^\star) \geq \frac{2^{j+1}}{n} \tilde{\lambda}_2(t)+ \frac{8 K_2 t}{n} \right) \\
&\leq (\lceil \log_2(\eta n \sqrt{p}) \rceil + 1) \exp(-t).
\end{align*}
Choosing $t = \log p$, we see that by the union bound
\begin{align*}
&\mathbb{P} \left( \exists \tilde{\beta} \in \mathcal{B} : Z(\tilde{\beta}, \beta^\star) \geq 2 \Vert \tilde{\beta} - \beta^\star \Vert_1 \tilde{\lambda}_2(\log p) + \frac{\tilde{\lambda}_2( \log p)}{n} + \frac{8 K_2 \log p}{n} \right)
\\
&\leq (\lceil \log_2(\eta n \sqrt{p}) \rceil + 2) \exp(- \log p).
\end{align*}
\end{proof}

\subsection{Proofs of the lemmas in Subsection \ref{ss:robustslope}}

\begin{proof}[Proof of Lemma \ref{lemma:slopel1}]
\begin{align*}
\frac{J_\lambda(\beta)}{\lambda_{p}} &= \frac{\lambda_1}{\lambda_p} \vert \beta \vert_{(1)} + \dots + \frac{\lambda_p}{\lambda_{p}} \vert \beta \vert_{(p)} \\
&\geq \Vert \beta \Vert_1.
\end{align*}
It follows that for all $w \in \mathbb{R}^p$
\begin{equation*}
J_\lambda (w)^* \leq \left( \lambda_{p} \Vert w \Vert_1 \right)^* = \underset{\beta \in \mathbb{R}^p: \Vert \beta \Vert_1 \leq 1}{\sup} \ \lambda_p \vert w^T \beta \vert = \lambda_p \underset{\beta \in \mathbb{R}^p: \Vert \beta \Vert_1 \leq 1}{\sup} \ \vert w^T \beta \vert = \lambda_p \Vert w \Vert_\infty.
\end{equation*}
\end{proof}

\begin{proof}[Proof of Lemma \ref{lemma:effectivesparsityslope}]
For $\Vert \beta_S \Vert_0 = s$ we have that
\begin{align*}
J_\lambda(\beta_S) = \sum_{j = 1}^s \lambda_j \vert \beta_S \vert_{(j)} \leq \lambda_1 \Vert \beta_S \Vert_1 \leq \lambda_1 \sqrt{s} \Vert \beta \Vert_2 \leq \lambda_1 \sqrt{s} \frac{\Vert \Sigma_X^{1/2} \beta_S \Vert_2}{\sqrt{\Lambda_{\min}(\Sigma_X)}}.
\end{align*}
Hence,
\begin{align*}
\frac{J_\lambda(\beta_S)}{\tau(\beta)} \leq \lambda_1 \sqrt{\frac{s}{\Lambda_{\min}(\Sigma_X)}}.
\end{align*}
\end{proof}